\documentclass{elsarticle}

\usepackage{multirow}
\usepackage{amsmath, amssymb, amsthm}
\usepackage{graphics}
\usepackage{algorithm}
\usepackage{algorithmic}
\usepackage{booktabs}
\usepackage[usenames, dvipsnames]{color}
\usepackage[all,cmtip]{xy}
\usepackage{mleftright}
\usepackage[normalem]{ulem}
\usepackage[abs]{overpic}

\usepackage{amsrefs}

\def\rmd{\textup{d}}
\def\rmT{\textup{T}}

\def\RR{\mathbb{R}}
\def\SS{\mathbb{S}}

\def\ZZ{\mathbb{Z}}

\def\CCC{\mathcal{C}}
\def\SSS{\mathcal{S}}

\def\PPP{\mathcal{P}}

\def\bfb{{\boldsymbol{b}}}
\def\bfc{{\boldsymbol{c}}}

\def\bff{{\boldsymbol{f}}}
\def\bfn{{\boldsymbol{n}}}
\def\bft{{\boldsymbol{t}}}
\def\bfx{{\boldsymbol{x}}}
\def\bfy{{\boldsymbol{y}}}

\def\bfF{{\boldsymbol{F}}}

\def\rmI{\mathrm{I}}
\def\rmII{\mathrm{II}}
\def\rmIII{\mathrm{III}}

\def\bfI{{\boldsymbol{I}}}
\def\bfM{{\boldsymbol{M}}}
\def\bfN{{\boldsymbol{N}}}
\def\bfQ{{\boldsymbol{Q}}}

\def\tbfc{\tilde{\boldsymbol{c}}}
\def\tr{\tilde{r}}

\newtheorem{theorem}{Theorem}
\newtheorem{corollary}[theorem]{Corollary}
\newtheorem{proposition}[theorem]{Proposition}
\newtheorem{lemma}[theorem]{Lemma}
\theoremstyle{definition}

\newtheorem{example}{Example}
\newtheorem{remark}{Remark}

\newcommand{\tagarray}{%
\mbox{}\refstepcounter{equation}%
$(\theequation)$%
}

\begin{document}

\begin{frontmatter}
\title{Symmetries of Canal Surfaces and Dupin Cyclides}

%%%%%%%%%%%%%%%%%%%%%%%%%%%%%%%%%%%%%%%%%%%%%%%%%%%%%%%%%%%%%%%%%%%%%%%%%%%%%%%%%%%%%%%%%%%%%%%%%
%% NOTA PARA MI: CUANDO VAYAMOS A ESCRIBIR PAPER, QUITAR COMENTARIOS DESDE AQUI...................

\author[a]{Juan Gerardo Alc\'azar\fnref{proy}}
\ead{juange.alcazar@uah.es}
\author[b]{Heidi E. I. Dahl}
\ead{Heidi.Dahl@sintef.no}
\author[b]{Georg Muntingh}
\ead{Georg.Muntingh@sintef.no}

\address[a]{Departamento de F\'{\i}sica y Matem\'aticas, Universidad de Alcal\'a,
E-28871 Madrid, Spain}
\address[b]{SINTEF ICT, PO Box 124 Blindern, 0314 Oslo, Norway}

\fntext[proy]{Supported by the Spanish ``Ministerio de
Ciencia e Innovacion" under the Project MTM2014-54141-P.
Member of the Research Group {\sc asynacs} (Ref. {\sc ccee2011/r34}) }

%%%% HASTA AQUI...............................................................
%%%%%%%%%%%%%%%%%%%%%%%%%%%%%%%%%%%%%%%%%%%%%%%%%%%%%%%%%%%%%%%%%%%%%%%%%%%%%%

\begin{abstract}
We develop a characterization for the existence of symmetries of canal surfaces defined by a rational spine curve and rational radius function. In turn, this characterization inspires an algorithm for computing the symmetries of such canal surfaces. For Dupin cyclides in canonical form, we apply the characterization to derive an intrinsic description of their symmetries and symmetry groups, which gives rise to a method for computing the symmetries of a Dupin cyclide not necessarily in canonical form. As a final application, we discuss the construction of patches and blends of rational canal surfaces with a prescribed symmetry.
\end{abstract}

\end{frontmatter}

\section{Introduction}

Symmetry is a feature commonly encountered in nature and in manufactured items. From an aesthetic point of view, it is usually associated with the notions of beauty and proportion. But in real objects it is often desirable because of physical reasons, too. As a consequence, Geometric Modeling employs many symmetric shapes as their building blocks, like planes, certain quadrics, surfaces of revolution or cylindrical surfaces, exhibiting different types of symmetry.

Canal surfaces, and Dupin cyclides as a distinguished subfamily, form a class of surfaces also used in Geometric Modeling. Canal surfaces are envelopes of 1-parameter families of spheres, whose radii are parametrized by a radius function and whose centers form a parametric spine curve. In fact, a canal surface is completely determined by this spine curve and radius function. These surfaces have been studied extensively during the last 20 years \cite{Cho1, Cho2, Dahl14, Dahl2014, Jia, Krasauskas07, Peternell.Pottman97, VL16} because of their applications in Computer Aided Geometric Design in general and, in particular, in operations like joining and blending \cite{Dahl14, Dahl2014, Krasauskas07}.

Dupin cyclides form a special family of rational canal surfaces, and they have also received much attention since their introduction \cite{Dupin}. In this case, the source of interest is two-fold: from a theoretical point of view, due to the fact that they are the only surfaces that are canal surfaces in more than one way \cite{Maxwell} and their remarkable geometric properties \cite{Chandru}; from a practical point of view, because of their applicability in joining and blending \cite{Dutta}. 

However, a generic canal surface is not necessarily symmetric. In this paper, we provide a characterization of symmetric canal surfaces, under the assumption that the spine curve and the radius function are rational. A similar result is provided for Dupin cyclides. Although our original motivation is theoretical, we have applied our results to constructing rational patches of canal surfaces with prescribed symmetries, and also to designing rational blends of canal surfaces with prescribed symmetries in certain cases. These two operations are useful in Computer Aided Design, when a graphical model must, for aesthetic or functional reasons, satisfy certain symmetries.

In order to conduct our study, we make use of ideas regarding symmetries and similarities of rational curves developed by two of the authors \cite{Alcazar.Hermoso.Muntingh14, Alcazar.Hermoso.Muntingh14b, Alcazar.Hermoso.Muntingh15, Alcazar.Hermoso.Muntingh16, Alcazar.Hermoso.Muntingh16a}. In particular, we relate the symmetries of the canal surface to symmetries of the spine curves, and to isometries (if any) between \emph{different} spine curves, in the case of Dupin cyclides. 

When applied to Dupin cyclides, our results provide a classification of the symmetries of these surfaces, together with their symmetry groups. Certainly, some results regarding the symmetries of Dupin cyclides are already known. For instance, it is well known --- and easy to see from the implicit equations of Dupin cyclides in canonical form --- that these surfaces are symmetric with respect to the planes containing the spine curves, and therefore with respect to the line intersecting these planes as well. Nevertheless, it is not always easy to systematically derive a classification of all symmetries of a surface from its implicit equation. As an illustration of our characterisation of symmetric canal surfaces and Dupin cyclides, we prove that in the generic situation the aforementioned symmetries are the only symmetries exhibited by Dupin cyclides. Furthermore, we identify the special cases where extra symmetries appear and then determine these symmetries, thus leading to a complete classification. 

The computations in this paper have been implemented in a \texttt{Sage} worksheet, which can be tried out online following a link on the website of the last author \cite{WebsiteGeorg}. 

\section{Preliminaries on canal surfaces, Dupin cyclides and symmetries} \label{sec:prelim}
\subsection{Background on canal surfaces}\label{sec:background}
A \emph{canal surface} $\SSS = \SSS_{\bfc,r} \subset \RR^3$ is the envelope of a 1-parameter family $\Sigma = \Sigma_{\bfc, r}$ of spheres, centered at a \emph{spine curve} $\bfc$ and with \emph{radius function} $r$. The defining equations for $\SSS$ are
\begin{subequations}\label{eq:Sigma1SigmaDot1}
\begin{align}
    \Sigma_{\bfc,r}(t) &: \|\bfx - \bfc(t)\|^2 - r^2(t) = 0,\label{eq:Sigma1}\\
\dot\Sigma_{\bfc,r}(t) &: \langle\bfx - \bfc(t), \dot\bfc(t)\rangle + r(t)\dot r(t) = 0.\label{eq:SigmaDot1}
\end{align}
\end{subequations}
For fixed $t$, the first equation is a sphere and the second is a plane, intersecting in the \emph{characteristic circle} $k(t) = k_{\bfc,r}(t) := \Sigma_{\bfc,r}(t)\cap\dot\Sigma_{\bfc,r}(t)$; see Figure \ref{fig:newcanal}.

It follows from \eqref{eq:Sigma1SigmaDot1} that $k(t)$ is nonempty (over the real numbers) precisely when $\|\dot\bfc(t)\|^2 \geq \dot r(t)^2$, degenerating to a single point if equality holds at $t$, and to a curve plus the 1-parameter family of tangent planes $\dot\Sigma(t)$ if equality holds in a nonempty interval \cite[\S 2]{Peternell.Pottman97}. To exclude these degenerate cases, we assume that $\|\dot\bfc(t)\|^2 > \dot r(t)^2$ holds for all parameters $t$ for which $\bfc(t)$ and $r(t)$ are defined, guaranteeing that \eqref{eq:Sigma1SigmaDot1} defines an irreducible real surface $\SSS$. For these statements and other general results on canal surfaces that we recall in this section, we refer the interested reader to \cite{Peternell.Pottman97, Dahl14, Dahl2014, Krasauskas07}.

\begin{figure}
\begin{center}
\includegraphics[scale=0.2]{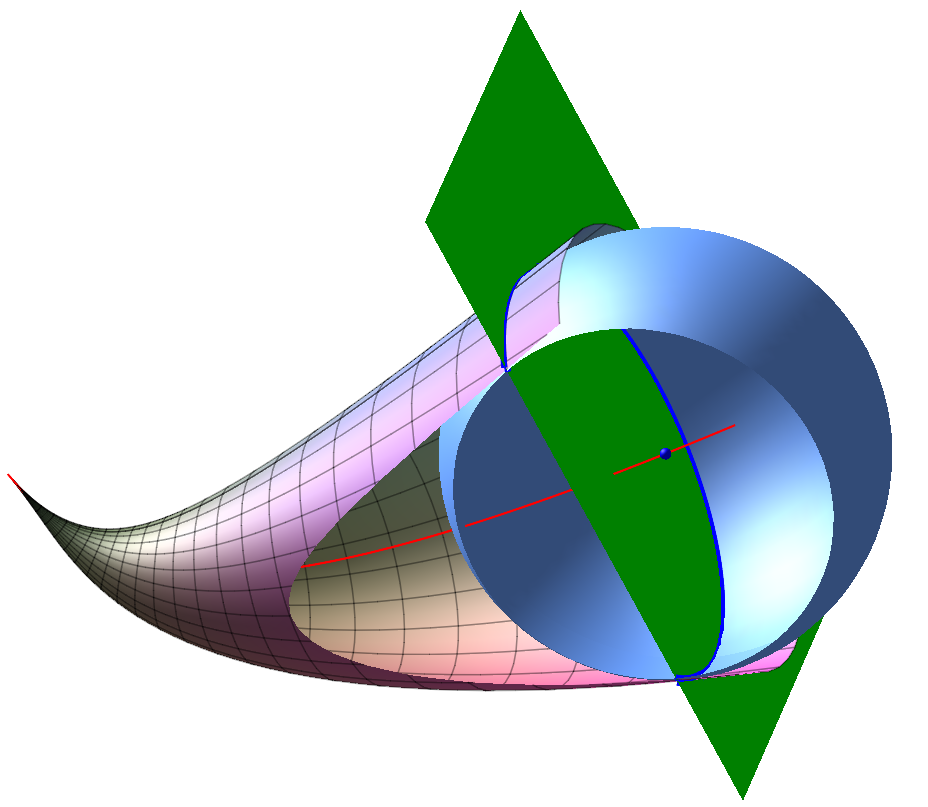}
\caption{A piece of a canal surface and a corresponding spine curve segment, with, for a given parameter $t_0$, the plane $\dot\Sigma(t_0)$ and the sphere $\Sigma(t_0)$ with center $\bfc(t_0)$ and radius $r(t_0)$, intersecting in the characteristic circle $k(t_0)$.}\label{fig:newcanal}
\end{center}
\end{figure}

\begin{remark}\label{rem:Minkowski}
The family $\Sigma$ of spheres can be identified with a curve in the $4$-dimensional Minkowski space $\RR^{3,1}$, where the point $\big(\bfc(t); r(t)\big)$ corresponds to a sphere $\Sigma(t)$ with center $\bfc(t)$ and radius $r(t)$. The sign of the radius gives us the \emph{orientation} of the sphere (towards the centre when $r(t) > 0$, and outwards when $r(t) < 0$), which is inherited by the canal surface. Thus an \emph{oriented} canal surface corresponds to a curve $\big(\bfc; r\big)$ in $\RR^{3,1}$. 
\end{remark}

The canal surface $\SSS$ admits parametrizations of the form
\begin{equation}\label{eq:CanalSurfaceParametrization}
\bfF (t,s) = \bfF_{\bfc,r} (t,s) := \bfc(t) + r(t) \bfN(t,s).
\end{equation}
Note that while \eqref{eq:Sigma1SigmaDot1} remains valid upon replacing $r$ by $-r$, the change in sign of the radius results in a change of orientation of the canal surface, i.e., a reversal of the direction of the unit normal vector field described by $\bfN (t, s)$. Thus the sign of the second term of the parametrization in \eqref{eq:CanalSurfaceParametrization} remains positive. 

Replacing $\bfx$ in \eqref{eq:Sigma1SigmaDot1} by \eqref{eq:CanalSurfaceParametrization} and factoring out $r(t)$ yields
\begin{subequations}\label{eq:Sigma2SigmaDot2}
\begin{align}
\|\bfN (t, s)\|^2 = 1,\label{eq:Sigma2}\\
\langle \bfN(t,s), \dot\bfc(t)\rangle + \dot r(t) = 0.\label{eq:SigmaDot2}
\end{align}
\end{subequations}
In fact, $\bfN (t, s)$ describes the unit normal vector field of $\SSS$. Also, for each characteristic circle $k(t)$, the normal lines to $\SSS$ along the circle intersect at the point $\bfc(t)$ on the spine curve. Moreover, the derivative $\dot \bfc(t)$ at this point is normal to the plane containing the characteristic circle.  

If $r$ is a constant, then we have a special type of canal surface, called a \emph{pipe} surface. Furthermore, if $\bfc$ is a line, then $\SSS$ is a surface of revolution. 

In this paper we assume that $\bfc$ and $r$ are real and rational and known. Additionally, we assume that $\bfc$ is \emph{proper}, i.e., birational or equivalently injective except for perhaps finitely many parameter values. Since $\bfc$ and $r$ are rational, finding a rational parametrization of type \eqref{eq:CanalSurfaceParametrization} reduces to finding a rational parametrization of $\bfN$. This was first accomplished in \cite{Peternell.Pottman97}, and minimal degree rational parametrizations were obtained in \cite{Krasauskas07}. As a result, the surface $\SSS$ is also rational, and therefore irreducible. Observe that the spine curve of $\SSS$ is irreducible as well, since it is rationally parametrized by $\bfc$.

\subsection{Background on Dupin cyclides}\label{sec:background2}
One can wonder if $\bfc$ and $r$ are unique for a given canal surface $\SSS$, or if there are surfaces that are canal surfaces in at least two different ways. This question was answered by Maxwell \cite{Maxwell}, who showed that Dupin cyclides are the only canal surfaces with the latter property. In fact, these surfaces are the envelope of exactly \emph{two} different 1-parameter families of spheres with distinct spine curves $\bfc_1, \bfc_2$ and radius functions $r_1, r_2$. Dupin \cite{Dupin, Hilbert.Cohn-Vossen52} originally defined \emph{cyclides}, now called \emph{Dupin cyclides}, as surfaces whose lines of curvatures are circles. Since then, a variety of alternative definitions has arisen \cite{Chandru}, which is the underlying reason for their value in a variety of applications.

Maxwell \cite{Maxwell} also showed that these spine curves must be conics lying in perpendicular planes and passing through each other's foci, yielding three different cases corresponding to the nature of the spine curves. For each of these cases one can provide \cite[\S 23.2]{Degen02} a canonical description of the spine curves, radius functions and implicit forms, shown in Table \ref{tab:canonicalform}. This form depends on certain parameters $a,b,c,f,g$ satisfying $f^2 = a^2 - b^2$. Notice that $a,c\neq 0$ in Type I, $a,b\neq 0$ and $f>0$ in Type II, and $g\neq 0$ in Type III. Note also that Type II degenerates to Type I in the limit $f\to 0$. We say that a Dupin cyclide is in \emph{canonical form}, if it is the zeroset of $F^\rmI$, $F^\rmII$ or $F^\rmIII$ in Table \ref{tab:canonicalform} for certain parameters $a,b,c,f,g$ satisfying the above constraints.

As Dupin cyclides have rational spine curves and radius functions, eliminating the parameter $t$ from \eqref{eq:Sigma1SigmaDot1} yields a description as the zeroset of a polynomial. The implicit forms in Table \ref{tab:canonicalform} show that the Dupin cyclides of Types I and II are quartics, while the Dupin cyclides of Type III are cubics. Furthermore, the Dupin cyclides of Type I are tori, and these surfaces are surfaces of revolution, since the characteristic circles are rotationally invariant about the axis traced out by $\bfc_2$.

For the Dupin cyclides in this paper we will assume that the spine curves $\bfc_1,\bfc_2$ and the radius functions $r_1,r_2$ are known, yielding alternative representations as in \eqref{eq:Sigma1SigmaDot1} for the same canal surface $\SSS_{\bfc_1,r_1} = \SSS_{\bfc_2,r_2}$.

\begin{table}
\begin{tabular*}{\columnwidth}{@{\extracolsep{\stretch{1}}}*{3}{l}}
\toprule
I: Quartic & spine curve: circle/line & radius function\\
\multirow{2}{*}{\includegraphics[width=6em]{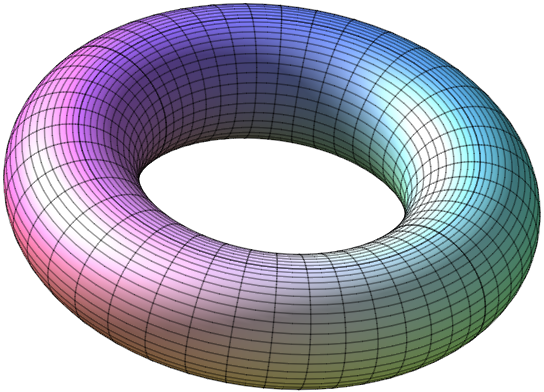}} & $\displaystyle{\bfc^\rmI_1(t)=a\left(\frac{1-t^2}{1+t^2},\frac{2t}{1+t^2},0\right)}$ & $r^\rmI_1(t)=c$\\ \vspace{0.5em}
 & $\displaystyle{\bfc^\rmI_2(t)=a\left(0,0,\frac{2t}{1-t^2}\right)}$ & $\displaystyle{r^\rmI_2(t)=c-a\frac{1+t^2}{1-t^2}}$\\
\multicolumn{3}{c}{$\displaystyle{0 = F^\rmI := (x^2+y^2+z^2+a^2-c^2)^2-4a^2(x^2+y^2)}$}\hfill \tagarray\label{eq:I}\\
\midrule
II: Quartic & spine curve: ellipse/hyperbola & radius function\\
\multirow{2}{*}{\includegraphics[width=6em]{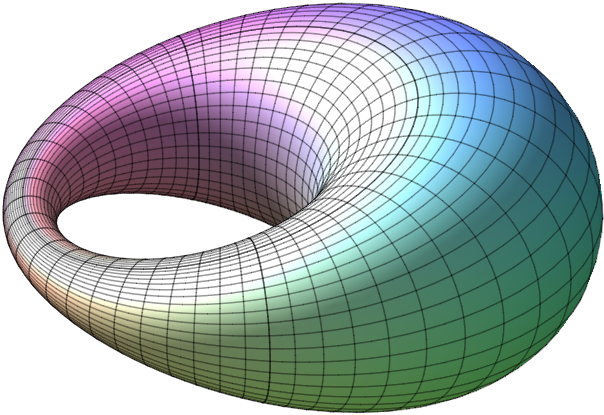}} & $\displaystyle{\bfc^\rmII_1(t)=\left(\frac{1-t^2}{1+t^2}a,\frac{2t}{1+t^2}b,0\right)}$ & $\displaystyle{r^\rmII_1(t)=c-f\frac{1-t^2}{1+t^2}}$ \\ \vspace{0.5em}
 & $\displaystyle{\bfc^\rmII_2(t)=\left(\frac{1+t^2}{1-t^2}f,0,\frac{2t}{1-t^2}b\right)}$ & $\displaystyle{r^\rmII_2(t)=c-a\frac{1+t^2}{1-t^2}}$\\
\multicolumn{3}{c}{$\displaystyle{0 = F^\rmII := (x^2+y^2+z^2+a^2-f^2-c^2)^2-4(ax-cf)^2-4y^2(a^2-f^2)}$} \hfill \tagarray\label{eq:II}\\
\midrule
III: Cubic & spine curve: parabola/parabola & radius function\\
\multirow{2}{*}{\includegraphics[width=6em]{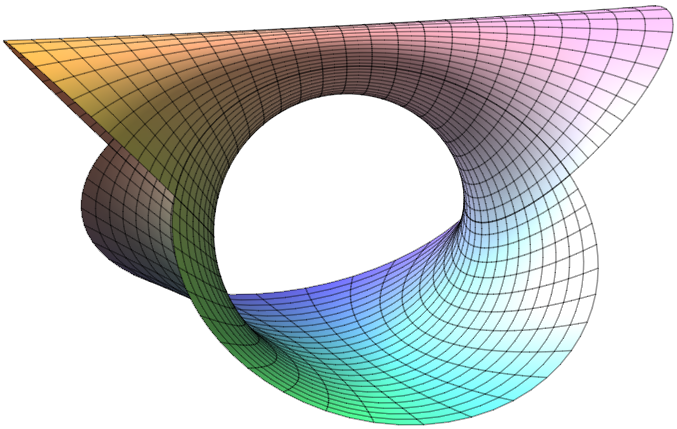}} & $\displaystyle{\bfc^\rmIII_1(t)= g\left(t^2-\frac{1}{2},2t,0\right)}$ & $\displaystyle{r^\rmIII_1(t)=c+g\left(t^2+\frac{1}{2}\right)}$\\ \vspace{0.5em}
 & $\displaystyle{\bfc^\rmIII_2(t)=g\left(\frac{1}{2}-t^2,0,2t\right)}$ & $\displaystyle{r^\rmIII_2(t)=c-g\left(t^2+\frac{1}{2}\right)}$\\
\multicolumn{3}{c}{$\displaystyle{0 = F^\rmIII := (x+c)(x^2+y^2+z^2)+(y^2-z^2)g-(g^2+c^2)x+(g^2-c^2)c}$} \hfill \tagarray\label{eq:III}\\
\bottomrule
\end{tabular*}
\caption{Parametric and corresponding implicit canonical forms for Dupin cyclides.}\label{tab:canonicalform}
\end{table}

\subsection{Symmetries of curves and surfaces}
An \emph{(affine) isometry $\bff$} of $\RR^n$ takes the form $\bff(\bfx) = \bfQ \bfx + \bfb$, with $\bfb\in \RR^n$ a vector and $\bfQ \in \RR^{n\times n}$ an orthogonal matrix, i.e., $\bfQ\bfQ^\rmT = \bfI$. If $\bff$ leaves a curve $\CCC\subset \RR^3$ (respectively a surface $\SSS\subset \RR^3$) invariant, in the sense that $\bff(\CCC) = \CCC$ (respectively $\bff(\SSS) = \SSS$), then $\bff$ is called a \emph{symmetry} of $\CCC$ (respectively $\SSS$). The identity map $\bff = \text{id}_{\RR^n}$ is referred to as the \emph{trivial} isometry/symmetry. A curve or surface is called \emph{symmetric} if it has a nontrivial symmetry.

The nontrivial isometries include reflections in a plane (mirror symmetries), rotations about an axis (axial symmetries), and translations, and these combine in commutative pairs to form twists, glide reflections, and rotatory reflections. A composition of three reflections in mutually perpendicular planes through a point yields a \emph{central symmetry} with respect to this point. The particular case of rotation by an angle $\pi$ is of special interest, and it is called a \emph{half-turn}. For a description of the types of nontrivial symmetries of Euclidean space, see \cite{Coxeter69} and \cite[\S 2]{Alcazar.Hermoso.Muntingh15}.

It is well known that the birational functions on the line are the \emph{M\"obius transformations} \cite{Sendra.Winkler.Perez-Diaz}, i.e., rational functions
\begin{equation}\label{eq:Moebius}
\varphi: \RR \dashrightarrow \RR, \qquad \varphi(t) = \frac{\alpha t + \beta}{\gamma t + \delta}, \qquad \alpha \delta - \beta \gamma \neq 0.
\end{equation}
The identity map $\varphi = \text{id}_\RR$ is referred to as the \emph{trivial} M\"obius transformation.

The following result is provided for rational plane and space curves in \cite{Alcazar.Hermoso.Muntingh14}, but the proof extends to any dimension, and in fact to the real analytic setting.

\begin{theorem}\label{thm:diagram}
Let $\CCC\subset \RR^n$ be a curve defined by a proper parametrization $\bfc: \RR\dashrightarrow \CCC\subset \RR^n$. The curve~$\CCC$ is symmetric if and only if there exists a nontrivial isometry $\bff$ and nontrivial M\"obius transformation $\varphi$ for which the diagram
\begin{equation}\label{eq:fundamentaldiagram}
\xymatrix{
\CCC \ar[r]^{\bff} & \CCC \\
\RR \ar@{-->}[u]^{\bfc} \ar@{-->}[r]_{\varphi} & \RR \ar@{-->}[u]_{\bfc}
}
\end{equation}
commutes.
\end{theorem}

We will say that the M\"obius transformation in Theorem \ref{thm:diagram} is \emph{associated} with the isometry $\bff$. Theorem \ref{thm:diagram} is used in \cite{Alcazar.Hermoso.Muntingh15} to reduce the computation of the symmetries of rational space curves to the computation of their associated M\"obius transformations, using the classical differential invariants (curvature and torsion) of space curves.

Suppose $\CCC$ is a rational space curve that is \emph{non-linear}, i.e., not a straight line. Then, for all but finitely many parameters $t$, the \emph{unit tangent}, \emph{principal normal}, and \emph{bi-normal} vectors of $\bfc$ at $t$,
\[
\bft_\bfc = \bft_\bfc(t) := \frac{\dot\bfc(t)}{\|\dot\bfc(t)\|},\quad
\bfn_\bfc = \bfn_\bfc(t) := \frac{\dot\bft_\bfc(t)}{\|\dot\bft_\bfc(t)\|}, \quad
\bfb_\bfc = \bfb_\bfc(t) := \bft_\bfc(t) \times \bfn_\bfc(t)
\]
are well defined and together form the \emph{Frenet frame} $(\bft_\bfc, \bfn_\bfc,\bfb_\bfc)$ of the curve at~$t$.

\begin{lemma}\label{lem:FrenetIsometry}
Let $\bff(\bfx) = \bfQ\bfx +\bfb$ be an isometry of $\RR^3$ and $\bfc$ a parametrized curve with Frenet frame $(\bft_\bfc, \bfn_\bfc, \bfb_\bfc)$. Then the curve $\bff\circ \bfc$ has Frenet frame
\[ (\bft_{\bff\circ \bfc}, \bfn_{\bff\circ \bfc}, \bfb_{\bff\circ \bfc}) = (\bfQ \bft_\bfc, \bfQ \bfn_\bfc, \det(\bfQ)\bfQ \bfb_\bfc).\]
\end{lemma}
\begin{proof}
One has
\begin{align*}
\bft_{\bff\circ \bfc} & = \left. \frac{\rmd \bff\circ \bfc}{\rmd t}  \right/ \left\|\frac{\rmd \bff\circ \bfc}{\rmd t}\right\| = \left.\bfQ \frac{\rmd \bfc}{\rmd t}\right/\left\|\bfQ \frac{\rmd \bfc}{\rmd t}\right\| = \bfQ \bft_\bfc,\\
\bfn_{\bff\circ \bfc} & = \left. \frac{\rmd \bft_{\bff\circ \bfc} }{\rmd t} \right/ \left\| \frac{\rmd \bft_{\bff\circ \bfc} }{\rmd t} \right\|
                     = \left. \bfQ \frac{\rmd \bft_\bfc }{\rmd t} \right/ \left\|\bfQ \frac{\rmd \bft_\bfc }{\rmd t} \right\| = \bfQ \bfn_\bfc,\\
\bfb_{\bff\circ \bfc} & = \bft_{\bff\circ \bfc} \times \bfn_{\bff\circ \bfc} = (\bfQ \bft_\bfc)\times (\bfQ \bfn_\bfc) = \det(\bfQ) \bfQ (\bft_\bfc \times \bfn_\bfc) = \det(\bfQ) \bfQ \bfb_\bfc,
\end{align*}
where we used the identity
\begin{equation}\label{eq:cross_product_orthogonal}
(\bfM \bfx) \times (\bfM\bfy) = \det\bfM\cdot \bfM^{-\rmT} (\bfx\times \bfy),
\end{equation}
which holds for any invertible matrix $\bfM\in \RR^{3\times 3}$ and vectors $\bfx, \bfy\in \RR^3$.
\end{proof}

The following result on real rational functions is necessary for establishing the characterization of symmetries of canal surfaces stated in Theorem \ref{thm:singlespine}. Since we were unable to find an appropriate reference, we include a short proof.

\begin{lemma} \label{ratfunc}
Let $g,h$ be two real rational functions. If $|g(t)|=|h(t)|$ holds for every $t\in \RR$ for which both sides are defined, then either $g = h$ or $g=-h$.
\end{lemma}
\begin{proof}
Since $g,h$ are real rational functions, there exists an open interval $I\subset \RR$ where $g,h$ are both defined and have constant sign. Therefore, for any $t\in I$, $|g(t)| - |h(t)|$ is defined, and $g(t)-h(t) = 0$ or $g(t)+h(t) = 0$ holds identically for $t\in I$. Assume that $g(t) - h(t) = 0$ holds identically for $t\in I$; the second possibility is analogous. Since the restriction of $g - h$ to $I$ is rational and well defined, it is analytic on its domain $I$. Therefore, if this restriction is zero in $I$, the Identity Theorem implies that $g = h$.
\end{proof}

\section{Symmetries of canal surfaces with a unique spine curve:\\ characterization and algorithm}\label{sec:sym}
In this section we characterize the existence of symmetries of a canal surface $\SSS$ that is not a Dupin cyclide. As an application of this characterization, we develop an algorithm for computing these symmetries.

If $\SSS$ has a linear spine curve $\bfc$, then it is a surface of revolution. Surfaces of revolution can be detected by using the methods in \cite{AG16, VL15}, and their symmetries essentially follow from those of the directrix curve \cite[\S 2.2.4]{AHM15}. Hence from now on we will assume that $\bfc$ is non-linear.

In order to solve our problem (in this and the next section) we recall that whenever $\bfc$ is non-linear, the Frenet frame of $\bfc$ is well defined and forms the basis of the following convenient (but in general nonrational) parametrization of the surface normals \cite[Equation (3.12)]{Dahl14},
\begin{equation}\label{eq:NormalsFrenet}
\bfN_{\bfc,r}(t, s) = \frac{\dot r(t)}{\|\dot\bfc(t)\|} \bft_\bfc + \sqrt{1 - \left(\frac{\dot r(t)}{\|\dot\bfc(t)\|}\right)^2} \left(\frac{1-s^2}{1+s^2} \bfn_\bfc + \frac{2s}{1+s^2} \bfb_\bfc \right).
\end{equation}

\subsection{General lemmas}
In this section we consider the effect of a symmetry $\bff$ of the canal surface $\SSS_{\bfc,r}$ on its spine curve $\bfc$. Allowing for the case that $\SSS_{\bfc,r}$ is a Dupin cyclide, these results will be applied both in this section and in Section \ref{sec:dupin}. In particular, the following lemma shows that $\tilde{\bfc} := \bff\circ \bfc$ is also a spine curve of $\SSS$, and the subsequent lemmas describe its relation to $\bfc$.

\begin{lemma} \label{lemma-1}
Let $\bff$ be a symmetry of $\SSS_{\bfc,r}$. Then $\SSS_{\bff\circ \bfc,\tilde{r}} = \SSS_{\bfc,r}$, where either $\tilde{r}=r$ or $\tilde{r}=-r$. In particular, $\bff\circ \bfc$ is also a spine curve of $\SSS_{\bfc,r}$.
\end{lemma}

\begin{proof} Let $\Sigma_{\bfc,r}$ be the 1-parameter family of spheres \eqref{eq:Sigma1} corresponding to the pair $(\bfc, r)$. 
For every $t$, the isometry $\bff$ maps
\[
\Sigma_{\bfc,r}(t) \overset{\bff}{\longrightarrow} \Sigma_{\bff\circ\bfc,\tilde{r}}(t), \qquad 
\dot\Sigma_{\bfc,r}(t) \overset{\bff}{\longrightarrow} \dot\Sigma_{\bff\circ\bfc,\tilde{r}}(t),
\]
where $\tilde{r}^2=r^2$, so by Lemma \ref{ratfunc} we have either $\tilde{r}=r$ or $\tilde{r}=-r$. The envelope of the spheres $\Sigma_{\bff\circ\bfc,\tilde{r}}(t)$ defines a canal surface $\SSS_{\bff\circ \bfc, \tilde{r}}$, and $\bff$ maps characteristic circles of $\SSS_{\bfc,r}$ to characteristic circles of $\SSS_{\bff\circ \bfc, \tilde{r}}$,
\[
k_{\bfc, r}(t) \overset{\bff}{\longrightarrow} k_{\bff\circ \bfc, \tilde{r}}(t).
\]
Moreover, since $\bff$ is a symmetry of $\SSS_{\bfc,r}$, the latter characteristic circles are contained in $\SSS_{\bfc,r}$. Since $\SSS_{\bfc,r}$ is irreducible, it follows that $\SSS_{\bfc, r} = \SSS_{\bff\circ\bfc, \tilde{r}}$. In particular $\bff$ maps each spine curve of $\SSS_{\bfc, r}$ to a spine curve of $\SSS_{\bfc, r}$.
\end{proof}

\begin{lemma}\label{lemma-2}
The spine curves $\bfc$ and $\tbfc$ have identical speed.
\end{lemma}
\begin{proof}
Since $\bfQ$ is orthogonal,
\begin{equation}\label{eq:mod}
  \left\| \dot{\tbfc} (t)\right\|
= \left\| \frac{\rmd}{\rmd t} (\bff\circ \bfc)(t) \right\|
= \|\bfQ \dot \bfc (t)\| = \|\dot \bfc(t)\|. \qedhere
\end{equation}
\end{proof}

\begin{lemma} \label{lemma-3}
The spine curves $\bfc$ and $\tbfc$, together with the radius function $r$, have surface normal parametrizations \eqref{eq:NormalsFrenet} related by
\begin{equation}\label{eq:lemma-3}
\bfQ\bfN_{\bfc,r}(t,s)=\bfN_{\tbfc, r} \big(t, \det(\bfQ) s\big).
\end{equation}
\end{lemma}

\begin{proof} The result follows after multiplying \eqref{eq:NormalsFrenet} by $\bfQ$, using Lemmas \ref{lem:FrenetIsometry} and \ref{lemma-2}.\!\!
\end{proof}

\subsection{Characterization}
Now we can state the characterization theorem for the existence of symmetries of canal surfaces with a single spine curve.

\begin{theorem}\label{thm:singlespine}
Let $\SSS_{\bfc,r}$ be a canal surface, not a Dupin cyclide, with non-linear spine curve $\bfc$. The isometry $\bff(\bfx) = \bfQ \bfx + \bfb$ is a symmetry of $\SSS_{\bfc,r}$ if and only if there exists a M\"obius transformation $\varphi$ such that
\begin{enumerate}
\item[C1:] the spine curve satisfies $\bff\circ \bfc = \bfc \circ \varphi$;
\item[C2:] the radius function satisfies $r^2 = (r\circ \varphi)^2$.
\end{enumerate}
\end{theorem}

\begin{proof}

``$\Longrightarrow$'' By Lemma \ref{lemma-1}, $\bff$ must be a symmetry of the spine curve $\bfc$. By Theorem \ref{thm:diagram} this is equivalent to the existence of a M\"obius transformation $\varphi$ for which $\bff\circ \bfc = \bfc \circ \varphi$, establishing C1.
Using Condition C1, then Lemma \ref{lemma-1} with $\tilde{r} = \pm r$, and finally that $\varphi$ is a birational map on the real line, one obtains
\[ \SSS_{\bfc \circ \varphi,\tilde{r}}
 = \SSS_{\bff \circ \bfc, \tilde{r}}
 = \SSS_{\bfc, r}
 = \SSS_{\bfc \circ \varphi,r\circ \varphi} \]
and deduces $r^2 = \tilde{r}^2 = (r\circ \varphi)^2$, establishing C2.

``$\Longleftarrow$'': Let $\bfF_{\bfc,r}(t,s)$ be the parametrization \eqref{eq:CanalSurfaceParametrization} of $\SSS_{\bfc,r}$ with normals $\bfN_{\bfc,r}(t,s)$ as in \eqref{eq:NormalsFrenet}. Since the radius function $r$ and $\tr := r\circ \varphi$ are real rational functions satisfying $r^2=\tilde{r}^2$, Lemma~\ref{ratfunc} implies $r=\pm\tr$. As a change of sign of the radius function results in a change of orientation of the canal surface and leaves the geometric shape unchanged, we may safely assume that $r=\tr$. Then using Lemma~\ref{lemma-3} with $\tbfc := \bfc\circ \varphi = \bff\circ \bfc$, the isometry $\bff$ maps any point $\bfF_{\bfc,r}(t,s)$ on $\SSS_{\bfc,r}$ to
\begin{align*}
\bff \circ \bfF_{\bfc,r}(t,s) & = (\bff\circ \bfc)(t) + r(t) \bfQ \bfN_{\bfc,r}(t,s) 
                                = \tbfc (t) + \tr(t) \bfN_{\tbfc,\tr}\big(t,\det(\bfQ) s\big)\\
                              & = \bfF_{\tbfc,\tr}\big(t,\det(\bfQ) s\big).
\end{align*}
It follows that $\bff \circ \bfF_{\bfc,r}(t,s)$ is a point on the canal surface with spine curve~$\tbfc$, radius function $\tr$ and corresponding normal $\bfN_{\tbfc,\tr}$ as in \eqref{eq:NormalsFrenet}, which is just a reparametrization of $\SSS_{\bfc,r}$. Therefore $\bff(\SSS_{\bfc,r})=\SSS_{\bfc,r}$, and $\bff$ is a symmetry of $\SSS_{\bfc,r}$.
\end{proof}

In the specific case of a pipe surface, Theorem \ref{thm:singlespine} takes the following form.

\begin{corollary}\label{cor:pipe}
Let $\SSS$ be a pipe surface, not a Dupin cyclide, with a non-linear spine curve $\bfc$. An isometry is a symmetry of $\SSS$ if and only if it is a symmetry of $\bfc$.
\end{corollary}

\subsection{Algorithm}\label{sec:algorithmsingle}
Let us cast the conditions in Theorem \ref{thm:singlespine} into a computer algebra setting. Each M\"obius transformation $\varphi$ can be described by introducing an auxiliary variable $u$ satisfying $u - \varphi(t) = 0$. Clearing denominators, we arrive at the \emph{M\"obius-like polynomial} $F(t,u) := u(\gamma t + \delta) - (\alpha t + \beta)$, which is zero precisely when $u = \varphi(t)$. Note that the M\"obius-like polynomials are the \emph{irreducible} bilinear polynomials, since $\alpha \delta - \beta \gamma \neq 0$. The M\"obius-like polynomial $F(t,u) = u - t$ is called \emph{trivial}, as the associated M\"obius transformation is the identity.

Under the condition $u = \varphi(t)$, Condition C2 of Theorem \ref{thm:singlespine} takes the form $r^2(t) - r^2(u) = 0$. Writing $r(t)=A(t)/B(t)$, with $A(t),B(t)$ coprime, and clearing denominators yields the polynomial condition
\begin{equation}\label{eq:R-pol0}
R(t,u) = A^2(t) B^2(u) - A^2(u) B^2(t) = 0.
\end{equation}
The following result shows how Condition C2 in Theorem \ref{thm:singlespine} can be tested by checking for M\"obius-like factors of $R$.

\begin{proposition}\label{thm:FdividesR}
Let $r=A/B$ be a real rational function, with corresponding bivariate polynomial $R$ as in \eqref{eq:R-pol0}, and let $\varphi$ be a M\"obius transformation with corresponding M\"obius-like polynomial $F$. Then Condition C2 holds precisely when $F$ divides $R$.
\end{proposition}
\begin{proof}
The equation $r^2(t) - r^2\big(\varphi(t)\big) = 0$ holds identically iff $R\big(t,\varphi(t)\big) = 0$ holds identically. In that case, the zeroset of $R$ contains the graph of $\varphi$, which is equal to the zeroset of $F$. Since $F$ is irreducible, it follows that this happens precisely when $F$ divides $R$.
\end{proof}

Combining Theorem \ref{thm:singlespine} and Proposition \ref{thm:FdividesR}, we observe that we can determine the symmetries of a canal surface by:
\begin{enumerate}
\item[(i)] testing Condition C2 on the radius function, by finding the M\"obius-like factors $F$ of $R$ and corresponding tentative M\"obius transformations $\varphi$;
\item[(ii)] testing Condition C1 on the spine curve for each such $\varphi$, by determining whether $\varphi$ corresponds to a symmetry $\bff$ of the spine curve $\bfc$.
\end{enumerate}

In step (i), since $\varphi$ does not necessarily have rational coefficients, we need to compute the factors of $R(t,u)$ with \emph{real algebraic} coefficients. This can be done for instance by using the {\tt AFactor} command in Maple 18, which works fine for moderate inputs. An alternative is to use the method of \cite[\S 3.2]{Alcazar.Hermoso.Muntingh15} to find {\it only} the M\"obius-like factors of $R(t,u)$, instead of a complete factorization.

In step (ii), one can apply the method in \cite[\S 4]{Alcazar.Hermoso.Muntingh15}, illustrated in Example~\ref{ex:crunode}, to determine an isometry $\bff$ associated with $\varphi$. Then one verifies Condition C1 by direct substitution.

Thus we arrive at Algorithm \texttt{SymCanal} for computing the symmetries of a canal surface with a unique non-linear spine curve. Notice that in this algorithm we distinguish pipe surfaces as a special case, since in that case $R$ is identically zero.

\begin{algorithm}[t!]
\begin{algorithmic}[1]
\REQUIRE A real, rational, properly parametrized, non-linear spine curve $\bfc$, and a real, rational radius function $r$, defining a parametrization \eqref{eq:CanalSurfaceParametrization} of a canal surface $\SSS$ that is not a Dupin cyclide.
\ENSURE The symmetries of $\SSS$.
\IF{$r$ is constant}
\STATE{return the symmetries of the spine curve by using the algorithm in \cite{Alcazar.Hermoso.Muntingh15}}
\ELSE
\STATE{find the M\"obius-like factors of $R$ and associated M\"obius transformations~$\varphi$}
\FOR{each M\"obius transformation $\varphi$}
\STATE{find the isometry $\bff$ associated with $\varphi$ using the method in \cite[\S 4]{Alcazar.Hermoso.Muntingh15}}
\STATE{if $\bff \circ \bfc=\bfc\circ \varphi$, return the isometry $\bff$}
\ENDFOR
\ENDIF
\end{algorithmic}
\caption{{\tt SymCanal}}\label{alg:symcanal}
\end{algorithm}

\begin{example}\label{ex:crunode}
Consider the crunode spine curve and radius function
\[ \bfc(t) = \left(\frac{t}{t^4 + 1},\frac{t^2}{t^4 + 1}, \frac{t^3}{t^4 + 1} \right), \qquad r(t) = \frac{t^2}{t^4 + 1}, \]
with corresponding canal surface shown in Figure \ref{fig:examples}, left. Then
\[ R(t,u) = (u - t)(u + t)(u t - 1)(u t + 1)(u^2t^2 + 1)(u^2 + t^2), \]
whose M\"obius-like factors correspond to M\"obius transformations
\[ \varphi_1(t) = t,\qquad \varphi_2(t) = -t,\qquad \varphi_3(t) = 1/t,\qquad \varphi_4(t) = -1/t.\]

Suppose Condition C1 holds with $\varphi = \varphi_1, \varphi_2$, i.e.,
\begin{equation}\label{eq-ex-3} \bfQ\bfc(t) + \bfb = \bff \circ \bfc(t) = \bfc\circ\varphi_i(t) = \bfc\big( (-1)^{i+1} t\big),\qquad i = 1, 2. \end{equation}
Substituting $t = 0$ we immediately determine $\bfb = 0$. Taking, for $n = 1,2,3$, the $n$-th derivative of \eqref{eq-ex-3} and evaluating at $t = 0$, we get $\bfQ$ from its action on $\{\bfc'(0),\bfc''(0),\bfc'''(0)\}$, and
\[ \bff_i(\bfx)=\bfQ_i\bfx,\qquad \bfQ_i = \begin{bmatrix} (-1)^{i+1} & 0 & 0\\ 0 & 1 & 0\\ 0 & 0 & (-1)^{i+1} \end{bmatrix},\quad i = 1, 2, \]
corresponding to the trivial symmetry and the half-turn around the $y$-axis. One directly verifies $\bff_i\circ \bfc(t)=\bfc\circ \varphi_i(t)$ for $i=1,2$, confirming that Condition C1 holds for $\bff_1,\bff_2$; so $\bff_1,\bff_2$ are symmetries of the canal surface defined by $(\bfc,r)$.

Next suppose Condition C1 holds with $\varphi = \varphi_3, \varphi_4$. With the above procedure, we obtain isometries
\[ \bff_i(\bfx)=\bfQ_i\bfx,\qquad \bfQ_i = \begin{bmatrix}0 & 0 & (-1)^{i+1}\\0 & 1 & 0\\(-1)^{i+1} & 0 & 0 \end{bmatrix},\quad i = 3,4,
\]
which are reflections in the planes $x + (-1)^i z = 0$. Verifying $\bff_i\circ \bfc(t) = \bfc \circ \varphi_i(t)$ for $i=3,4$, we confirm that $\bff_3$ and $\bff_4$ are also symmetries of the canal surface defined by $(\bfc,r)$.
\end{example}

\begin{figure}
\begin{center}
\includegraphics[scale=0.18]{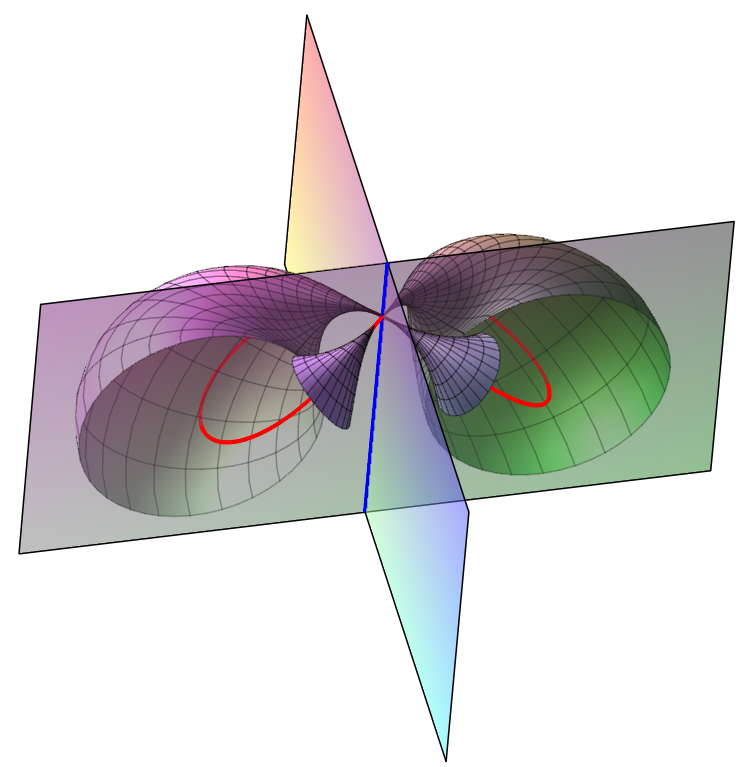}
\includegraphics[scale=0.21]{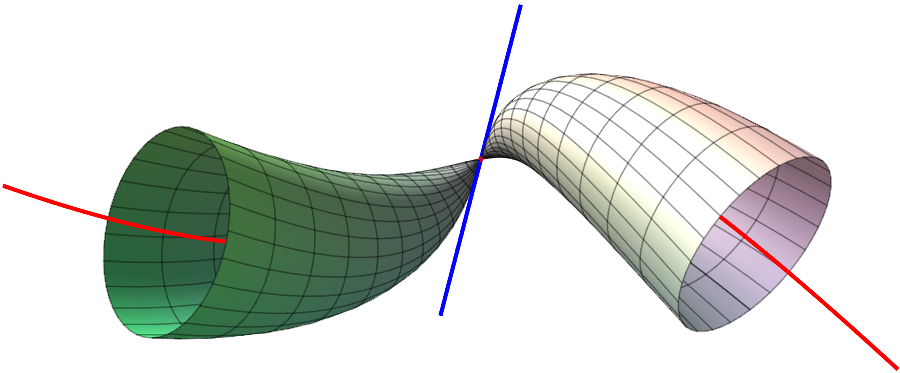}
\end{center}
\caption{Canal surfaces together with their spine curves and symmetry elements for Example~\ref{ex:crunode} (left) and Example~\ref{ex:twisted} (right).}\label{fig:examples}
\end{figure}

\section{Symmetries of Dupin cyclides:\\characterization, classification and algorithm}\label{sec:dupin}
In this section we consider the remaining case of a Dupin cyclide $\SSS$ with two distinct spine curves $\bfc_1,\bfc_2$ and corresponding radius functions $r_1,r_2$. First we provide a characterization theorem for an isometry to be a symmetry of a Dupin cyclide. Based on this theorem, we provide a complete classification of the symmetries of the three types of Dupin cyclides, together with the symmetry group in each case. Finally, based on this classification, we present an algorithm for computing the symmetries of a Dupin cyclide represented by pairs $(\bfc_i,r_i)$, with $i=1,2$, not necessarily in canonical form.

\subsection{Characterization} \label{sec:char-dupin}
Using Lemmas \ref{lemma-1}--\ref{lemma-3}, we establish the following characterization theorem for the symmetries of Dupin cyclides. 

\begin{theorem}\label{thm:MainTheorem}
Let $\SSS$ be a Dupin cyclide with non-linear spine curves $\bfc_1,\bfc_2$ and radius functions $r_1,r_2$. The isometry $\bff(\bfx) = \bfQ \bfx + \bfb$ is a symmetry of $\SSS$ if and only if there exist M\"obius transformations $\varphi_1,\varphi_2$ such that either
\begin{enumerate}
\item[A1:] the spine curves satisfy $\bff\circ \bfc_1 = \bfc_1 \circ \varphi_1$ and $\bff\circ \bfc_2 = \bfc_2 \circ \varphi_2$;
\item[A2:] the radius functions satisfy $r_1^2 = (r_1\circ\varphi_1)^2$ and $r_2^2 = (r_2\circ\varphi_2)^2$,
\end{enumerate}
or
\begin{enumerate}
\item[B1:] the spine curves satisfy $\bff\circ \bfc_1 = \bfc_2 \circ \varphi_1$ and $\bff\circ \bfc_2 = \bfc_1 \circ \varphi_2$;
\item[B2:] the radius functions satisfy $r_1^2 = (r_2\circ\varphi_1)^2$ and $r_2^2 = (r_1\circ\varphi_2)^2$.
\end{enumerate}
\end{theorem}

\begin{proof}
``$\Longrightarrow$'': By Lemma \ref{lemma-1}, for $i=1,2$ we have that $\bff\circ \bfc_i$ must also be a spine curve of $\SSS$. Suppose $\bff$ maps one of the spine curves of $\SSS$, say $\bfc_1$, to itself. Since $\bff$ is a bijection, it follows that $\bff$ also maps $\bfc_2$ to itself. Thus $\bff$ is a symmetry of both $\bfc_1$ and $\bfc_2$, and by Theorem \ref{thm:diagram} this is equivalent to the existence of M\"obius transformations $\varphi_1$ and $\varphi_2$ for which $\bff\circ \bfc_i = \bfc_i \circ \varphi_i$, $i = 1, 2$, establishing A1. Then, for $i=1,2$, A2 is established as in the implication ``$\Longrightarrow$'' of Theorem \ref{thm:singlespine}. 

Now let $\CCC_1, \CCC_2$ be the curves defined by $\bfc_1, \bfc_2$ and suppose that $\bff$ maps $\CCC_1$ to $\CCC_2$. Then, since $\bff$ is a bijection, it also maps $\CCC_2$ to $\CCC_1$ by Lemma \ref{lemma-1}. In particular $\bff^2$ is a symmetry of both $\CCC_1$ and $\CCC_2$. By Theorem 2 in \cite{Alcazar.Hermoso.Muntingh15} and Theorem 9 in \cite{Alcazar.Hermoso.Muntingh16} there exist M\"obius transformations $\varphi_1,\varphi_2$ such that $\bff\circ \bfc_1 = \bfc_2 \circ \varphi_1$ and $\bff\circ \bfc_2 = \bfc_1 \circ \varphi_2$, establishing B1. Using first Condition B1, then Lemma 4 with $r = r_1$ and $\tilde{r}= \tilde{r}_1 = \pm r_1$, next that $(\bfc_1, r_1)$ and $(\bfc_2,r_2)$ both define $\SSS$, and finally that $\varphi_1$ is a birational map on the real line, one obtains
\[ \SSS_{\bfc_2\circ\varphi_1,\tilde{r}_1}
 = \SSS_{\bff\circ \bfc_1,\tilde{r}_1}
 = \SSS_{\bfc_1,r_1}
 = \SSS_{\bfc_2,r_2}
 = \SSS_{\bfc_2\circ\varphi_1,r_2\circ\varphi_1} \]
and deduces $r^2_1 = \tilde{r}^2_1 = (r_2\circ \varphi_1)^2$, establishing B2.

``$\Longleftarrow$'': Let $i,j \in \{1, 2\}$ and suppose $i=j$ (resp. $i\neq j$). Let $\bfF_{\bfc_i,r_i}(t,s)$ be the parametrization \eqref{eq:CanalSurfaceParametrization} of $\SSS$ with normals $\bfN_{\bfc_i,r_i}(t,s)$ as in \eqref{eq:NormalsFrenet}. Since the radius function $r_i$ and $\tilde{r}_j := r_j\circ \varphi_i$ are real rational functions satisfying $r^2_i=\tilde{r}_j^2$ by A2 (resp. B2), Lemma \ref{ratfunc} implies $r_i=\pm\tr_j$. As a change of sign of the radius function results in a change of orientation of the canal surface and leaves the geometric shape unchanged, we may safely assume that $r_i=\tr_j$. Then using Lemma \ref{lemma-3} with $\bfc = \bfc_i$, $\tilde{\bfc}= \tbfc_j := \bfc_j \circ \varphi_i = \bff \circ \bfc_i$, $r = r_i = \tilde{r}_j$, we get $\bfQ\bfN_{\bfc_i,r_i}(t,s)=\bfN_{\tbfc_j, \tr_j} \big(t, \det(\bfQ) s\big)$, and the proof proceeds as in the implication ``$\Longleftarrow$'' of Theorem \ref{thm:singlespine}. 
\end{proof}

\begin{remark}
Notice that $\varphi_1$ is typically not equal to $\varphi_2$. By Theorem \ref{thm:diagram}, Case~A of Theorem \ref{thm:MainTheorem} implies that $\bff$ is a symmetry of each spine curve of $\SSS$; in particular, in this case $\bff$ maps each spine curve to itself. Case B implies that the spine curves $\CCC_1$ and $\CCC_2$ are mapped to each other by an isometry $\bff$ that is not a symmetry of either $\CCC_1$ or $\CCC_2$. Nevertheless, $\bff^2$ is a symmetry of both $\CCC_1$ and $\CCC_2$ with associated M\"obius transformations $\varphi_2\circ \varphi_1$ and $\varphi_1\circ\varphi_2$.
\end{remark}

Analogous to the method in Section \ref{sec:algorithmsingle}, we can use Conditions A2 and B2 to compute the symmetries of the surface. Again introducing the auxiliary variable $u$ satisfying $u - \varphi(t) = 0$, Conditions A2 and B2 of Theorem \ref{thm:MainTheorem} take the form $r^2_i(t) - r^2_j(u) = 0$, with $(i,j) = (1,1), (2,2)$ for Condition A2 and $(i,j) = (1,2), (2,1)$ for Condition B2. Clearing denominators yields corresponding polynomial conditions
\begin{equation}\label{eq:R-pol}
R_{ij}(t,u) := A^2_i(t) B^2_j(u) - A^2_j(u) B^2_i(t) = 0,\qquad i,j=1,2,
\end{equation}
where $r_1 = A_1/B_1$ and $r_2 = A_2/B_2$, with $(A_1, B_1)$ and $(A_2,B_2)$ pairs of coprime univariate polynomials. The following proposition, which has a proof analogous to that of Proposition \ref{thm:FdividesR}, shows how the tentative M\"obius transformations can be determined in the form of M\"obius-like factors of the polynomials $R_{ij}$.

\begin{proposition}\label{thm:FdividesR2} For $i, j = 1,2$, let $r_i = A_i/B_i$ be real rational functions, with corresponding bivariate polynomials $R_{ij}$ as in \eqref{eq:R-pol}, and let $\varphi_1,\varphi_2$ be M\"obius transformations with corresponding M\"obius-like polynomials $F_1, F_2$. Then
\begin{itemize}
\item[A2] holds precisely when $F_1$ divides $R_{11}$ and $F_2$ divides $R_{22}$;
\item[B2] holds precisely when $F_1$ divides $R_{12}$ and $F_2$ divides $R_{21}$.
\end{itemize}
\end{proposition}
 
\begin{example}\label{ex:SymmetriesIII}
In the proof of Theorem \ref{thm:SymmetriesIII}, it is derived that the M\"obius transformations $\varphi_\pm(t) := \pm t$ satisfy Condition B2 in Theorem \ref{thm:MainTheorem} for a Dupin cyclide of Type III with $c = 0$. Let us determine the associated symmetries in Case B using the method in \cite{Alcazar.Hermoso.Muntingh15}.

Let $\bff(\bfx) = \bfQ \bfx + \bfb$, and suppose that $\bff\circ \bfc^\rmIII_1 = \bfc^\rmIII_2 \circ \varphi_\pm$. Then
\begin{equation}\label{eq:fundrel1}
\bfQ f\begin{bmatrix} t^2 - \frac12\\ 2t\\ 0\end{bmatrix} + \bfb
= \bfQ \bfc^\rmIII_1(t) + \bfb
= \bff\circ \bfc^\rmIII_1 (t)
= \bfc^\rmIII_2(\pm t)
= f\begin{bmatrix} \frac12 - t^2\\ 0\\ \pm 2 t\end{bmatrix}.
\end{equation}
Differentiating once and twice, evaluating at $t = 0$, and taking the cross product,
\begin{equation}\label{eq:findQ}
\bfQ \begin{bmatrix}  0\\ 1\\ 0\end{bmatrix} = \begin{bmatrix}   0\\ 0\\ \pm 1\end{bmatrix},\quad
\bfQ \begin{bmatrix}  1\\ 0\\ 0\end{bmatrix} = \begin{bmatrix} - 1\\ 0\\ 0\end{bmatrix}, \quad
\bfQ \begin{bmatrix} 0 \\ 0\\ 1\end{bmatrix}
= \begin{bmatrix} 0 \\ \pm \det\bfQ \\ 0\end{bmatrix},
\end{equation}
where we used the identity \eqref{eq:cross_product_orthogonal}. Applying the rules of matrix block multiplication to \eqref{eq:findQ}, we obtain
\[ \bfQ 
= \begin{bmatrix} -1 & 0 & 0 \\ 0 & 0 & \pm \det \bfQ\\ 0 & \pm 1 & 0 \end{bmatrix}. \]
Substituting this $\bfQ$ and $t = 0$ in \eqref{eq:fundrel1} yields $\bfb = 0$, and it follows that $\bff(x,y,z) = (-x, \pm \det\bfQ z, \pm y)$. Moreover, $\bff\circ \bfc^\rmIII_2 = \bfc^\rmIII_1 \circ \varphi_{\pm\det\bfQ}$,
so that Conditions B1 and B2 are satisfied for the Cases (i)--(l) in Table \ref{tab:symmetries}.
\end{example}

\subsection{Classification}
In this subsection we apply the preceding results and ideas to classify the symmetries of Dupin cyclides, providing the symmetry groups in each case. For this purpose, we assume each Dupin cyclide to be in canonical form, i.e., with spine curves $\bfc_1^\alpha, \bfc_2^{\alpha}$ and radius functions $r_1^\alpha, r_2^\alpha$ for one of the Types $\alpha = \rmI, \rmII$ or $\rmIII$ as in Table~\ref{tab:canonicalform}, and corresponding implicit equation $F^\alpha$. The results are summarized in Tables \ref{tab:symmetries} and \ref{tab:DupinSymmetryGroups}.

\begin{figure}
\begin{center}\includegraphics[scale=0.25]{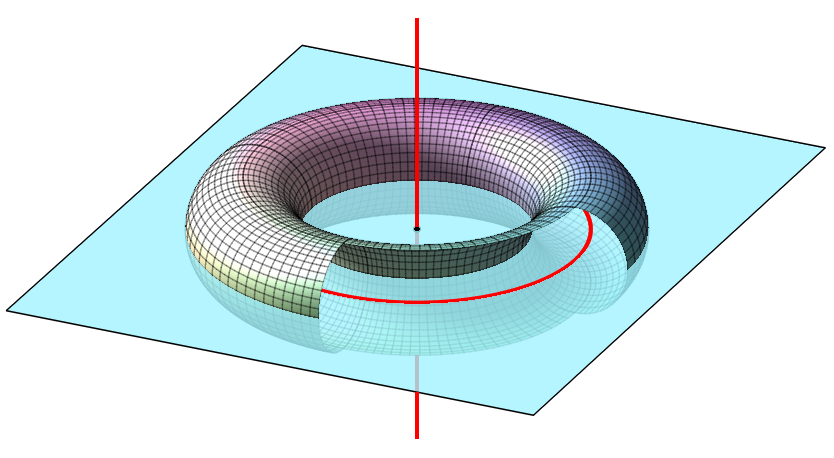}\end{center}
\caption{The Dupin cyclide $\SSS$ of Type I with parameters $a = 2$ and $c = 1$, together with its spine curves, the line $\bfc^\rmI_2$ passing through the barycenter of the circle $\bfc^\rmI_1$. The axis of the rotational symmetry coincides with the spine curve $\bfc^\rmI_2$, and the plane of $\bfc^\rmI_1$ corresponds to a mirror symmetry of $\SSS$.}\label{fig:torus}
\end{figure}

\subsubsection{Type I}
Let us address first the symmetries of the Dupin cyclides of Type I, i.e., tori. This can be deduced from results on symmetries of surfaces of revolution \cite[\S 2.2.4]{AHM15}. However, we will show that the results in this paper can also be used to easily derive these symmetries.

Even though one of the spine curves of the Dupin cyclide of Type I is a line, the following theorem shows that Theorem \ref{thm:MainTheorem} extends to these surfaces as well and explicitly describes the symmetries and associated M\"obius transformations.

\begin{theorem}\label{thm:SymmetriesI}
The isometry $\bff(\bfx) = \bfQ\bfx + \bfb$ is a symmetry of the Dupin cyclide $\SSS$ of Type I in Table \ref{tab:canonicalform} if and only if there exist M\"obius transformations $\varphi_1, \varphi_2$ such that Conditions A1 and A2 in Theorem \ref{thm:MainTheorem} hold. Moreover, with signs $\varepsilon_1,\varepsilon_2 \in \{-1, 1\}\simeq \ZZ_2$ and angle $\theta\in [0, 2\pi) \simeq \SS^1$, these M\"obius transformations take the form
\begin{equation}\label{eq:torus-Moebius} \varphi_1(t) = -\left(\frac{\cos(\theta/2)t + \sin(\theta/2)}{\sin(\theta/2)t-\cos(\theta/2) }\right)^{\varepsilon_2}, \qquad \varphi_2(t) = \varepsilon_1 t,
\end{equation}
with associated symmetries
\begin{equation}\label{eq:f_Dupin1}
\bff(\bfx) = \bfQ\bfx,\qquad \bfQ = 
\begin{bmatrix} \varepsilon_2 \cos\theta & -\varepsilon_2\sin\theta & 0 \\ \sin \theta & \cos\theta & 0\\ 0 & 0 & \varepsilon_1 \end{bmatrix}
\end{equation}
forming a symmetry group isomorphic to $\ZZ_2^2 \times \SS^1$.
\end{theorem}

\begin{proof}
We first determine the isometries $\bff$ and associated M\"obius transformations $\varphi_1, \varphi_2$ satisfying Conditions A1, A2. Substituting $r^\rmI_2$ into \eqref{eq:R-pol} yields
\begin{align*}
R^\rmI_{22}(t,u) & = 4a(u + t)(u - t) (t^2u^2a + t^2u^2c - t^2c - u^2c - a + c).
\end{align*}
Since $a,c\neq 0$, the right-most factor in $R^\rmI_{22}$ does not split, and the only M\"obius-like factors are $u-t$ and $u+t$, corresponding to M\"obius transformations $\varphi_2 = \pm t$ satisfying Condition A2 by Proposition \ref{thm:FdividesR2}. The relation $\bff \circ \bfc_2 = \bfc_2 \circ \varphi_2$ implies that $\bff$ satisfies $\bff(0,0,z) = (0,0,\pm z)$ for any point $(0,0,z)$ on the line $\bfc^\rmI_2$, i.e.,
\[
\bff(x,y,z) =
\left[ \begin{array}{c|c}
\widehat{\bfQ} & 0\\ \hline
0 & \pm 1
\end{array} \right]
\begin{bmatrix} x\\ y \\ z\end{bmatrix},\qquad
\widehat{\bfQ}\widehat{\bfQ}^{\rmT} = \bfI.
\]

On the other hand, since $r^\rmI_1$ is constant, Condition A2 does not determine (or even restrict) $\varphi_1$. However, it is well known that any orthogonal matrix $\widehat{\bfQ}\in \RR^{2,2}$ maps the circle to itself. This can be shown explicitly using the trigonometric reparametrization $\bfc^\rmI_1 \big(\tan(\phi/2)\big) = a(\cos \phi, \sin\phi, 0)$ and representation
\[
\widehat{\bfQ} \in \left\{
\begin{bmatrix} \varepsilon_2 \cos\theta & - \varepsilon_2 \sin\theta\\ \sin \theta & \cos\theta \end{bmatrix}
\,:\, \varepsilon_2\in \{\pm 1\},\ \theta\in [0,2\pi)\right\}. \]
Thus $\bff$ necessarily takes the form \eqref{eq:f_Dupin1} and corresponds to the reparametrizations $\phi\longmapsto \phi + \theta$ and $\phi \longmapsto \pi - \phi - \theta$, or equivalently to the M\"obius transformations \eqref{eq:torus-Moebius}. A direct calculation shows that such $\varphi_1, \varphi_2$ and $\bff$ satisfy Conditions A1 and A2.

``$\Longrightarrow$'': Using in addition that the circle $\bfc_1^\rmI$ and line $\bfc_2^\rmI$ are not related by an isometry, this is established as in Theorem \ref{thm:MainTheorem}.

``$\Longleftarrow$'': Using the above explicit form of $\bff$, a straightforward calculation shows that these indeed are symmetries of $\SSS$.
\end{proof}

\begin{remark}
Notice that the matrix $\bfQ$ in \eqref{eq:f_Dupin1} can be decomposed as a product 
\[
\begin{bmatrix} \varepsilon_2 \cos\theta & -\varepsilon_2\sin\theta & 0 \\ \sin \theta & \cos\theta & 0\\ 0 & 0 & \varepsilon_1 \end{bmatrix}=\begin{bmatrix} \varepsilon_2  & 0 & 0 \\ 0 & 1 & 0\\ 0 & 0 & 1 \end{bmatrix}\begin{bmatrix}  \cos\theta & -\sin\theta & 0 \\ \sin \theta & \cos\theta & 0\\ 0 & 0 & 1 \end{bmatrix}\begin{bmatrix} 1  & 0 & 0 \\ 0 & 1 & 0\\ 0 & 0 & \varepsilon_1 \end{bmatrix}.
\]
This shows that the symmetries of cyclides of Type I are the rotations about the line $\bfc_2^\rmI$, together with the composition of these rotations with the reflection in the plane containing the circle $\bfc_1^\rmI$ and$/$or with the reflection in any plane containing the line $\bfc_2^\rmI$.
\end{remark}

\begin{table}
\begin{tabular*}{\columnwidth}{@{}@{\extracolsep{\stretch{1}}}*{3}{c}p{0.388\columnwidth}c@{}}
\toprule
Case A & $(\varphi_1, \varphi_2)$ & $\bff(x,y,z)$ & description of the symmetry & Type\\
\midrule
(a) & $(+t, +t)$ & $(+x,+y,+z)$ & trivial symmetry & II, III\\
(b) & $(+t, -t)$ & $(+x,+y,-z)$ & reflection in the plane $\Pi_1$ & II, III\\
(c) & $(-t, +t)$ & $(+x,-y,+z)$ & reflection in the plane $\Pi_2$ & II, III\\
(d) & $(-t, -t)$ & $(+x,-y,-z)$ & half-turn about the line $\Pi_1\cap \Pi_2$ & II, III\\ \midrule
(e) & $(+\frac{1}{t}, -\frac{1}{t})$ & $(-x,+y,+z)$ & reflection in the plane $\Pi_0$ & II, $c = 0$\vspace{0.2em}\\ 
(f) & $(+\frac{1}{t}, +\frac{1}{t})$ & $(-x,+y,-z)$ & half-turn about the line $\Pi_0\cap \Pi_1$ & II, $c = 0$\vspace{0.2em}\\ 
(g) & $(-\frac{1}{t}, -\frac{1}{t})$ & $(-x,-y,+z)$ & half-turn about the line $\Pi_0\cap \Pi_2$ & II, $c = 0$\vspace{0.2em}\\
(h) & $(-\frac{1}{t}, +\frac{1}{t})$ & $(-x,-y,-z)$ & central symmetry about $O$ & II, $c = 0$ \\
\midrule
Case B & $(\varphi_1, \varphi_2)$ & $\bff(x,y,z)$ & description of the symmetry & Type\\
\midrule
(i) & $(+t, +t)$ & $(-x, +z, +y)$ & half-turn about the line $\Pi_0\cap \Pi_3$ & III, $c = 0$\\
(j) & $(-t, -t)$ & $(-x, -z, -y)$ & half-turn about the line $\Pi_0\cap \Pi_4$ & III, $c = 0$\\
(k) & $(+t, -t)$ & $(-x, -z, +y)$ & composition of a reflection in the plane $\Pi_0$ and a quarter-turn about the line $\Pi_1\cap\Pi_2$ & III, $c = 0$\\
(l) & $(-t, +t)$ & $(-x, +z, -y)$ & composition of a reflection in the plane $\Pi_0$ and a quarter-turn about the line $\Pi_1\cap\Pi_2$ & III, $c = 0$\\
\bottomrule
\end{tabular*}
\caption{For the Types II and III of Dupin cyclides in Table \ref{tab:canonicalform}, the table lists the discrete symmetries $\bff$ and M\"obius transformations $(\varphi_1, \varphi_2)$ associated with $\bff$ via Case A and Case~B in Theorem~\ref{thm:MainTheorem}. Here $\Pi_0: x = 0$ is the plane through $O := (0,0,0)$, the barycenter of the ellipse and hyperbola and average of the foci (which are, as a set, also the vertices) of the parabolas, and perpendicular to the planes $\Pi_1, \Pi_2$ containing the spine curves $\bfc_1,\bfc_2$; $\Pi_3 : z - y = 0, \Pi_4 : z + y = 0$ are the bisector planes of $\Pi_1, \Pi_2$.}\label{tab:symmetries}
\end{table}

\begin{table}
\begin{tabular*}{\columnwidth}{l@{\extracolsep{\stretch{1}}}*{4}{c}}
\toprule
Type & graphic & parameters & symmetries & group\\
\midrule
I: Quartic
& \raisebox{-1em}{\includegraphics[scale=0.07]{Torus}}
& \text{all} & $\{$(a),(b)$\}\times \ZZ_2\times \SS^1$ & $\ZZ_2^2 \times \SS^1$ \\
\midrule
II: Quartic
& \raisebox{-1em}{\includegraphics[scale=0.05]{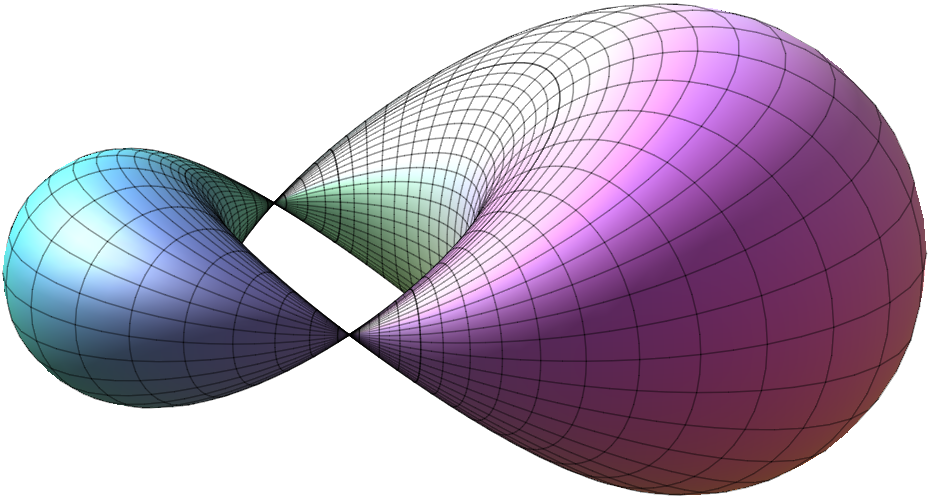}}
& $c\neq 0$ & (a)--(d) & $\ZZ_2^2$\\
& \raisebox{-1em}{\includegraphics[scale=0.05]{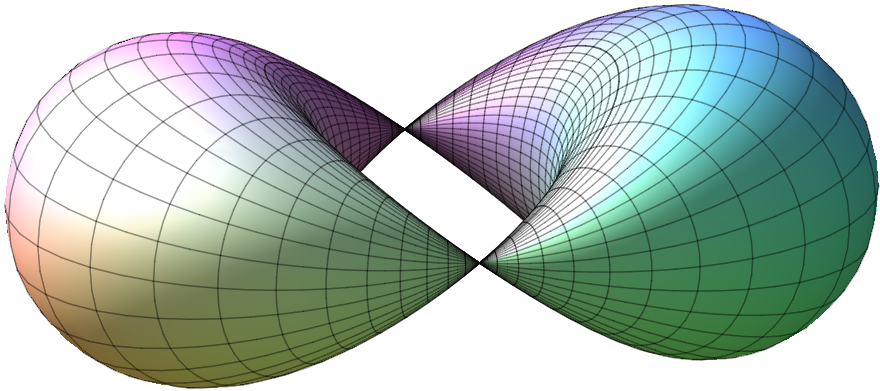}}
& $c  =  0$ & (a)--(h) & $\ZZ_2^3$\\
\midrule
III: Cubic
& \raisebox{-1em}{\includegraphics[scale=0.055]{Symmetric_cubic_cyclide_1_03}}
& $c\neq 0$ & (a)--(d) & $\ZZ_2^2$ \\
& \raisebox{-1em}{\includegraphics[scale=0.055]{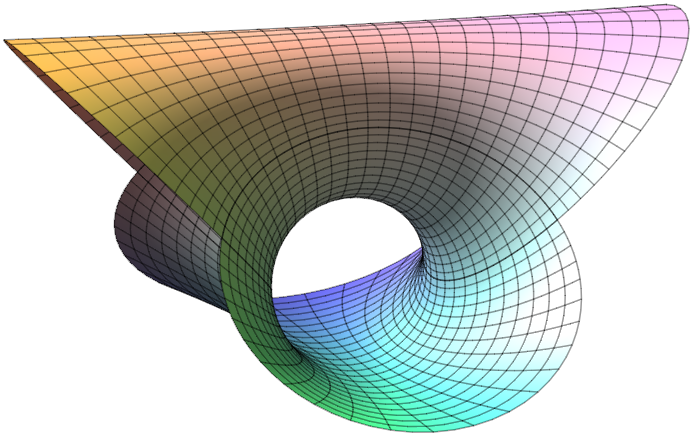}}
& $c  =  0$ & (a)--(d), (i)--(l) & $D_4$ \\
\bottomrule
\end{tabular*}
\caption{The symmetries and symmetry groups for the Dupin cyclides of Types I, II, III in Table \ref{tab:canonicalform} with parameters $a,b,c,f,g$. Here $\ZZ_2\times \SS^1$ is identified with the subgroup of $O(3)$ of rotations about the $z$-axis and reflections in planes containing the $z$-axis.}\label{tab:DupinSymmetryGroups}
\end{table}

\subsubsection{Type II}
Now we apply Theorem \ref{thm:MainTheorem} to derive the symmetries of Dupin cyclides of Type II. Every such Dupin cyclide has an ellipse and hyperbola as its spine curves, which are not related by an isometry. It follows that for these Dupin cyclides Case B in Theorem \ref{thm:MainTheorem} cannot happen. We distinguish cases according to whether the parameter $c$ vanishes; see Figure \ref{fig:TypeIIsuper}.

\begin{theorem}\label{thm:SymmetriesII}
For any Dupin cyclide $\SSS$ of Type II in Table \ref{tab:canonicalform}, Conditions A1 and A2 in Theorem \ref{thm:MainTheorem} are satisfied if and only if: 
\begin{itemize}
\item $c\neq 0$ and $(\varphi_1, \varphi_2)$ and $\bff$ are given by Cases (a)--(d) in Table \ref{tab:symmetries}, forming a symmetry group isomorphic to $\ZZ_2^2$, the Klein four group;
\item $c  =  0$ and $(\varphi_1, \varphi_2)$ and $\bff$ are given by Cases (a)--(h) in Table \ref{tab:symmetries}, forming a symmetry group isomorphic to $\ZZ_2^3$, the elementary abelian group of order~8.
\end{itemize}
\end{theorem}
\begin{proof}
Inserting the radius functions $r^\rmII_1, r^\rmII_2$ from Table \ref{tab:canonicalform} into \eqref{eq:R-pol} yields
\begin{align*}
R^\rmII_{11}(t,u) & = -4f(u-t)(u+t)(cu^2t^2+fu^2t^2+cu^2+ct^2-f+c),\\
R^\rmII_{22}(t,u) & = +4a(u-t)(u+t)(au^2t^2+cu^2t^2-cu^2-ct^2-a+c),
\end{align*}
each of which has the M\"obius-like factors $F_1(t,u) = u-t$ and $F_2(t,u) = u+t$.
In each case the remaining factor has degree two, and whether it splits into two additional M\"obius-like factors
\[ F_3(t,u)=u(\gamma t+\delta)-(\alpha t+\beta),\qquad F_4(t,u)=u(\gamma't+\delta')-(\alpha't+\beta')\]
depends on the parameters $a,c,f$.

In particular for the polynomial $R^\rmII_{11}$, if the remaining factor satisfies
\[cu^2t^2+fu^2t^2+cu^2+ct^2-f+c=u^2\cdot [(c+f)t^2+c]+(ct^2+c-f)=F_3(t,u)F_4(t,u),\]
comparing the coefficients of $u$ on each side yields
\[(\gamma t+\delta)(\alpha't+\beta')=-(\gamma' t+\delta')(\alpha t+\beta),\]
implying that the M\"obius transformations corresponding to $F_3(t,u)$ and $F_4(t,u)$ are opposite. Hence, after an appropriate scaling of $F_3$ or $F_4$, the remaining factor satisfies
\[ u^2\cdot [(c+f)t^2+c]+(ct^2+c-f) = u^2\cdot (\gamma t+\delta)^2 - (\alpha t+\beta)^2.\]
Comparing the coefficients of $t$ and the coefficients of $u^2t$ yields $c = 0$, in which case
\[R^\rmII_{11}(t,u)|_{c = 0} = -4f^2(u-t)(u+t)(ut-1)(ut+1).\]
A similar argument shows that the remaining factor of $R^\rmII_{22}$ only factors when $c = 0$, in which case
\[R^\rmII_{22}(t,u)|_{c = 0} = 4a^2(u-t)(u+t)(ut-1)(ut+1).\]
Thus $R^\rmII_{11}$ and $R^\rmII_{22}$ each determine tentative M\"obius transformations $\varphi_1(t) = t, \varphi_2(t) = -t$, and two additional M\"obius transformations $\varphi_3(t) = 1/t, \varphi_4(t) = -1/t$ if and only if $c = 0$.

When $c\neq 0$, we therefore only get the M\"obius transformations $\varphi_1(t)=t, \varphi_2(t) = -t$,
for both $\bfc^\rmII_1$ and $\bfc^\rmII_2$, which combine in pairs associated with four potential symmetries. As in Example~\ref{ex:SymmetriesIII}, the nature of these symmetries can be determined by using the techniques in \cite{Alcazar.Hermoso.Muntingh15}, yielding Cases (a)--(d) in Table \ref{tab:symmetries}. When $c=0$, we get the M\"obius transformations $\varphi_1,\varphi_2,\varphi_3,\varphi_4$ for both $\bfc^\rmII_1$ and $\bfc^\rmII_2$, which combine in pairs associated with 16 potential symmetries. Proceeding as in Example \ref{ex:SymmetriesIII}, one finds that only 8 of these correspond to a symmetry of~$\SSS$, namely Cases (a)--(h) in Table \ref{tab:symmetries}.

Noting from the explicit representations in Table \ref{tab:symmetries} that the symmetries (b)--(h) have order 2, the second part follows from comparing to a list of groups of small order \cite[p. 85]{Rotman95}.
\end{proof}

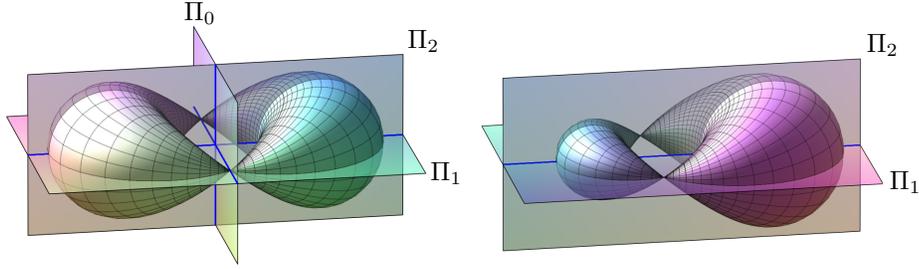
\begin{figure}
\begin{center}
\begin{overpic}[scale=0.18,unit=0.005\columnwidth]{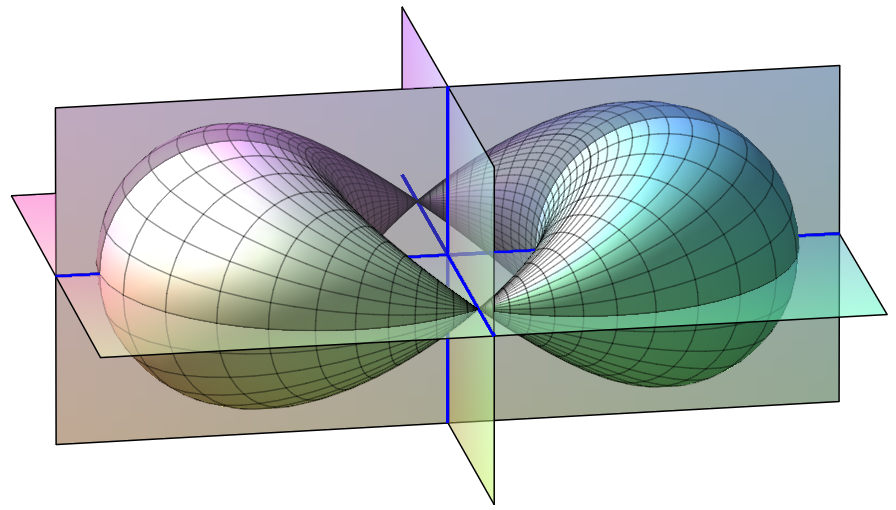}
\put(40,54){$\Pi_0$}
\put(94,19){$\Pi_1$}
\put(89,48){$\Pi_2$}
\end{overpic}\qquad
\begin{overpic}[scale=0.16, unit=0.005\columnwidth]{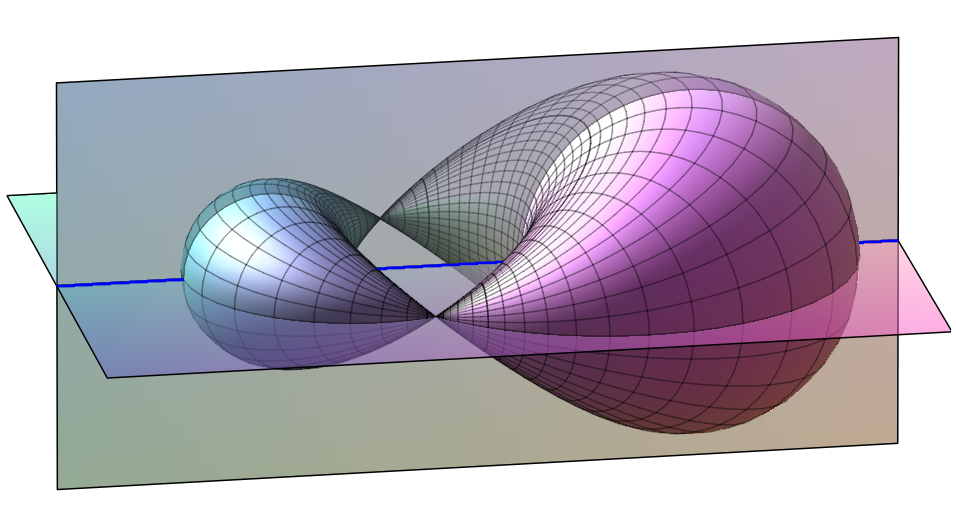}
\put(90,17){$\Pi_1$}
\put(85,47){$\Pi_2$}
\end{overpic}
\end{center}
\caption{The Dupin cyclide of Type II in Table \ref{tab:canonicalform} with parameters $a = 5, b = 4, c = 0, f = 3$ (left) and parameters $a = 5, b = 4, c = -1, f = 3$ (right)}\label{fig:TypeIIsuper}
\end{figure}

\subsubsection{Type III}
Next we apply Theorem \ref{thm:MainTheorem} to derive the symmetries of Dupin cyclides of Type III. Every such Dupin cyclide has parabolas as its spine curves, which might be related by an isometry. Hence, for such Dupin cyclides, it is necessary to analyse Case B of Theorem \ref{thm:MainTheorem} as well. We distinguish cases according to whether the parameter $c$ vanishes; see Figure \ref{fig:TypeIIIsuper}.

\begin{theorem}\label{thm:SymmetriesIII}
For any Dupin cyclide $\SSS$ of Type III in Table \ref{tab:canonicalform}:
\begin{itemize}
\item Conditions A1 and A2 in Theorem \ref{thm:MainTheorem} are satisfied if and only if $(\varphi_1, \varphi_2)$ and $\bff$ are given by Cases (a)--(d) in Table \ref{tab:symmetries}.
\item Conditions B1 and B2 in Theorem \ref{thm:MainTheorem} are satisfied if and only if $c = 0$ and $(\varphi_1, \varphi_2)$ and $\bff$ are given by Cases (i)--(l) in Table \ref{tab:symmetries}.
\end{itemize}
In particular:
\begin{itemize}
\item If $c\neq 0$, the symmetries of $\SSS$ are (a)--(d), forming a group isomorphic to $\ZZ_2^2$, the Klein four group.
\item If $c  =  0$, the symmetries of $\SSS$ are (a)--(d) and (i)--(l), forming a group isomorphic to $D_4$, the dihedral group of order eight.
\end{itemize}
\end{theorem}
\begin{proof}
\emph{Case A:} Inserting the radius functions $r^\rmIII_1, r^\rmIII_2$ from Table \ref{tab:canonicalform} into \eqref{eq:R-pol} yields
\begin{align*}
R^\rmIII_{11}(t,u) & = -16g(u+t)(u-t)\big(g (u^2 + t^2 + 1) + 2c\big),\\
R^\rmIII_{22}(t,u) & = -16g(u+t)(u-t)\big(g (u^2 + t^2 + 1) - 2c\big),
\end{align*}
each of which has the M\"obius-like factors $F_1(t,u) = u-t$ and $F_2(t,u) = u+t$. As the remaining factor is irreducible in each case, we find that the tentative M\"obius transformations are $\varphi_1(t) = t$ and $\varphi_2(t) = -t$ for both $\bfc^\rmIII_1$ and $\bfc^\rmIII_2$. These again combine to four potential symmetries, corresponding exactly to the Cases (a)--(d) in Table \ref{tab:symmetries}.

\emph{Case B:} One computes
\[ R^\rmIII_{12}(t,u)=(r^\rmIII_1)^2(t)-(r^\rmIII_2)^2(u)= -16 g (gu^2 - gt^2 - 2c)(u^2 + t^2 + 1). \]
The factor $u^2+t^2+1$ is irreducible. If $c\neq 0$, then the factor $-gu^2+gt^2+2c$ defines a hyperbola, since $g\neq 0$, and is therefore also irreducible. It follows that $R^\rmIII_{12}$ does not have M\"obius-like factors, so that there are no symmetries corresponding to Case B of Theorem \ref{thm:MainTheorem} when $c\neq 0$. However, if $c=0$, then $R^\rmIII_{12}$ has the M\"obius-like factors $F_\pm(t,u) = u \mp t$, corresponding to the M\"obius transformations $\varphi_\pm(t) = \pm t$. 
From Example \ref{ex:SymmetriesIII} it follows that Conditions B1 and B2 are satisfied for the Cases (i)--(l) in Table \ref{tab:symmetries}.

Noting from the explicit representations in Table \ref{tab:symmetries} that the symmetries (b), (c), (d), (i), (j) have order 2 and (k), (l) have order 4, the second part follows from comparing to a list of groups of small order \cite[p. 85]{Rotman95}.
\end{proof}

\begin{figure}
\begin{center}
\begin{overpic}[width=0.46\columnwidth,unit=0.005\columnwidth]{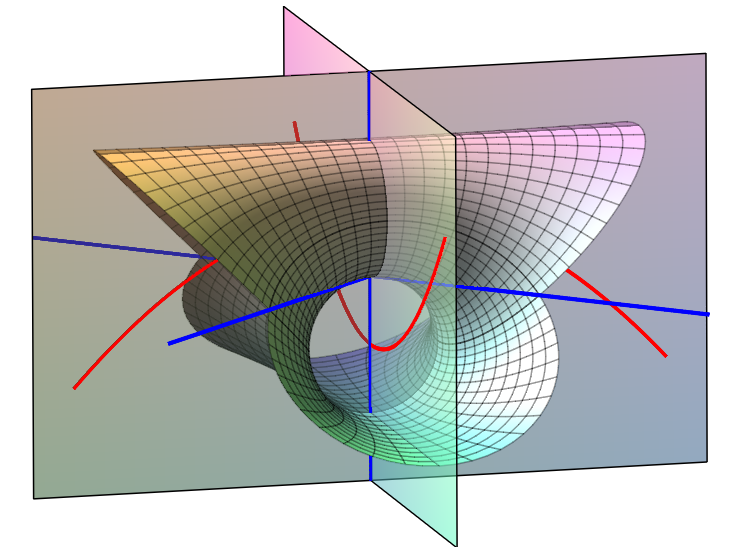}
\put(88,63){$\Pi_1$}
\put(33,69){$\Pi_2$}
\end{overpic}\qquad
\begin{overpic}[width=0.46\columnwidth, unit=0.005\columnwidth]{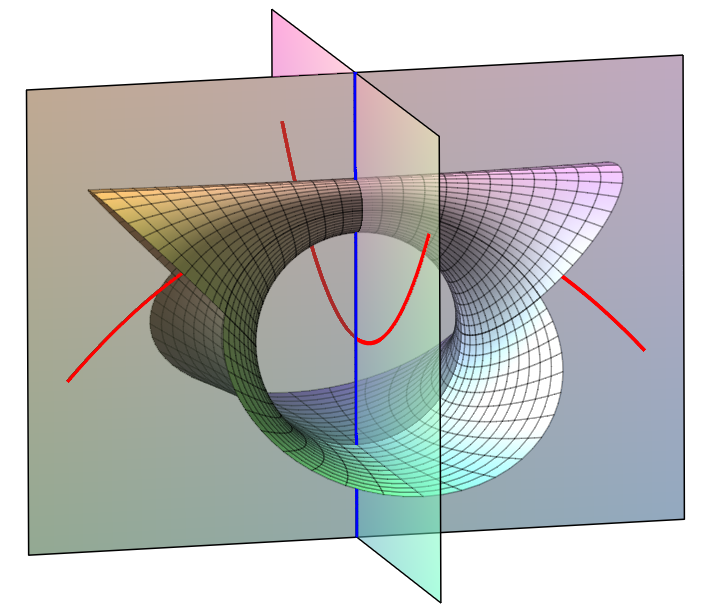}
\put(88,73){$\Pi_1$}
\put(33,79){$\Pi_2$}
\end{overpic}
\end{center}
\caption{The Dupin cyclide of Type III in Table \ref{tab:canonicalform} with parameters $c = 0, g = 1$ (left) and parameters $c = 0.3, g = 1$ (right).}\label{fig:TypeIIIsuper}
\end{figure}

While it is well known that any cyclide of Type $\rmII$ or $\rmIII$ is symmetric with respect to the planes containing each of the spine curves, and therefore also with respect to the intersection line of these two planes, the preceding results show that when $c\neq 0$, such a cyclide cannot have any other symmetry. In fact, we have proven that cyclides of Type $\rmII$ and $\rmIII$ have either 4 or 8 symmetries; in this last case, we say that it is a \emph{super-symmetric cyclide} (see Figure \ref{fig:TypeIIsuper}, left, and Figure \ref{fig:TypeIIIsuper}, left).

\subsection{Algorithm}
We end with providing an algorithm for computing the symmetries of a Dupin cyclide $\SSS$ defined by spine curves $\bfc_1, \bfc_2$ and corresponding radius functions $r_1, r_2$, not necessarily given in canonical form. Whether the Dupin cyclide is of Type I, II, or III follows from the nature of the conics; this is easily determined, for instance by implicitization or by computing the curvature, which is independent of position and orientation. Moreover, we have the following result.

\begin{lemma}\label{lem-supersym}
A Dupin cyclide $\SSS = \SSS_{\bfc_1,r_1}=\SSS_{\bfc_2,r_2}$ is super-symmetric if and only if one of the following cases holds:
\begin{itemize}
\item $\SSS$ is of Type II, and the radius function corresponding to the ellipse has minimum $r_{\min}$ and maximum $r_{\max}$ satisfying $r_{\min} + r_{\max} = 0$.
\item $\SSS$ is of Type III, and $r_1 + r_2 = 0$ holds identically.
\end{itemize}
\end{lemma}

\begin{proof}
Since the conditions $r_{\min}+r_{\max}=0$ and $r_1+r_2=0$ remain valid under reparametrization, we can assume that $(\bfc_1, r_1)$ and $(\bfc_2, r_2)$ take the canonical forms in Table \ref{tab:canonicalform}. In the first case of the lemma,
\[ r_{\min} = r^\rmII_1(0) = c-f, \qquad r_{\max} = \lim_{t\to \pm \infty} r^\rmII_1(t)=c+f, \]
while in the second case
\[ r_1 = r^\rmIII_1 = c + g\left(t^2 + \frac12 \right),\qquad r_2 = r^\rmIII_2 = c - g\left(t^2 + \frac12 \right).\]
In either case the sum is zero precisely when $c = 0$, which proves the lemma.
\end{proof}

We can now sketch a method for determining the symmetries of~$\SSS$.
\begin{enumerate}
\item \emph{Determine the Type of $\SSS$ from the nature of its spine curves}.
\item \emph{Determine the invariants of the conics referenced in Table \ref{tab:symmetries}}.\\ 
Find planes $\Pi_1, \Pi_2$ containing $\bfc_1, \bfc_2$. Next determine $O$, i.e., for
\begin{itemize}
\item Type I, the barycenter of the circle;
\item Type II, the barycenter of the ellipse/hyperbola;
\item Type III, the average of the vertex and focal point, for each parabola.
\end{itemize}
Let $\Pi_0$ be the plane passing through $O$ and perpendicular to $\Pi_1, \Pi_2$. For Type III, we also determine the bisector planes $\Pi_3,\Pi_4$ of $\Pi_1,\Pi_2$.
\item \emph{For Types II and III, determine if $\SSS$ is super-symmetric using Lemma~\ref{lem-supersym}.}
\item \emph{Look up the symmetry groups and symmetries of $\SSS$ in Tables \ref{tab:symmetries} and~\ref{tab:DupinSymmetryGroups}.}
\end{enumerate}

\section{Canal surfaces and blending patches with prescribed symmetries} \label{sec-applications}

In this section, we first construct (patches of) canal surfaces with a prescribed planar, axial or central symmetry, which are the most common symmetries in Computer Graphics and Geometric Design. After this we address the computation of blendings with a prescribed symmetry of two non-intersecting canal surfaces, under certain conditions. In either case we apply Theorem \ref{thm:singlespine}, choosing a symmetric rational spine curve (Condition C1) and a rational radius function that respects the symmetries of the spine curve (Condition C2).

\subsection{Building symmetric (patches of) canal surfaces}\label{building-1}

In order to build a patch of a rational canal surface $\SSS_{\bfc,r}$ with a prescribed symmetry $\bff$, we first need to find a rational spine curve $\bfc$ invariant under $\bff$. For this purpose, consider the Bernstein polynomials
\[ B_{i,n}(t) = {n\choose i} t^i (1-t)^{n-i}, \qquad i = 0,\ldots, n, \]
of some fixed degree $n$, and pick a B\'ezier curve with parametrization
\begin{equation}\label{eq:Bezier}
\bfc(t)=\sum_{i=0}^n \bfb_iB_{i,n}(t),\qquad t\in [0,1],
\end{equation}
where the control points $\bfb_0, \ldots, \bfb_n$ form a control polygon $\PPP$ invariant under~$\bff$. In this situation, it is known that $\bfc$ is invariant under the symmetry $\bff$ \cite{SanchezReyes}. Hence, by Condition C1 of Theorem \ref{thm:singlespine} there exists a M\"obius transformation $\varphi$ such that
\[
\sum_{i=0}^n \bff(\bfb_i) B_{i,n}(t) = \bff \circ \bfc(t) = \bfc \circ \varphi(t)
= \sum_{i=0}^n \bfb_iB_{i,n}(\varphi(t)).
\]
Using that $B_{i,n}(1 - t) = B_{n-i,n}(t)$, one obtains two cases:
\begin{align*}
\text{Case 1} & : & \varphi(t) & = t,     & \bff(\bfb_i) & = \bfb_i, & i = 0,\ldots,n;\\
\text{Case 2} & : & \varphi(t) & = 1 - t, & \bff(\bfb_i) & = \bfb_{n-i}, & i = 0,\ldots,n.
\end{align*}
In the first case, Condition C2 of  Theorem \ref{thm:singlespine} holds trivially. In the second case, we require a radius function $r(t)$ satisfying $r^2(t) = (r\circ \varphi)^2(t) = r^2(1-t)$, which is achieved using the following result.

\begin{theorem}\label{thm:rBezier}
A degree $n$ polynomial $r$ satisfies $r^2(t) = r^2(1-t)$ if and only if
\begin{equation}\label{eq:Bezier2}
r(t)=\sum_{i=0}^n a_i B_{i,n}(t), \qquad a_i = (-1)^n a_{n-i},\qquad i=0,\ldots,n.
\end{equation}
\end{theorem}
\begin{proof}
``$\Longrightarrow$": By Lemma~\ref{ratfunc}, $r(t)=\pm r(1-t)$. Expressing $r$ in the Bernstein basis and using that $B_{i,n}(1 - t) = B_{n-i,n}(t)$, it follows
\[ 0 = \sum_{i=0}^n a_i B_{i,n}(t) \mp \sum_{i=0}^n a_i B_{n-i,n}(t) = \sum_{i=0}^n (a_i\mp a_{n-i}) B_{i,n}(t),  \]
implying either $a_i = a_{n-i}$ for $i = 0,\ldots,n$, or $a_i = - a_{n-i}$ for $i = 0,\ldots, n$. Let $\Delta^k$ denote the $k$-th forward difference operator recursively defined by
\begin{equation}\label{eq:forwarddifference}
\Delta^0 a_i = a_i,\qquad \Delta^k a_i = \Delta^{k-1} a_{i+1} - \Delta^{k-1} a_i,\quad i = 0,\ldots, n-k, \quad k\geq 1.
\end{equation}
It follows by induction that
\[ \frac{\rmd^n r}{\rmd t^n} = n! \Delta^n a_0 = n! \sum_{j = 0}^n (-1)^j {n \choose j} a_{n-j}, \]
which is zero when $n$ is odd and $a_i = a_{n-i}$, or when $n$ is even and $a_i = - a_{n-i}$. Since $r$ has degree $n$, it takes the form \eqref{eq:Bezier2}.

``$\Longleftarrow$": By hypothesis, the coefficients $a_i$ in the Bernstein form \eqref{eq:Bezier2} of $r$ satisfy $a_i = (-1)^n a_{n-i}$. Using that $B_{i,n}(1-t)=B_{n-i,n}(t)$, it follows that $r(1-t)=(-1)^n \cdot r(t)$.
\end{proof}

Thus we can apply Theorem \ref{thm:singlespine} to construct a canal surface $\SSS_{\bfc,r}$ with the prescribed symmetry $\bff$. 
	
\begin{example}\label{ex:planarsymmetry}
The case of a symmetry $\bff$ with respect to a plane $\Pi$ is especially simple, as it holds to assume that $\bfc$ is contained in $\Pi$. In this case the restriction $\bff|_\Pi$ is the identity, and Conditions C1 and C2 in Theorem~\ref{thm:singlespine} are trivially satisfied taking $\varphi(t)=t$.
\end{example}

\begin{example}\label{ex:twisted}
Let $n = 3$, and consider the B\'ezier curve $\bfc$ as in \eqref{eq:Bezier} with
\[
\bfb_0 = \begin{bmatrix} -1 \\ 1 \\ -1 \end{bmatrix},\qquad
\bfb_1 = \begin{bmatrix} -\frac13\\-\frac13 \\ 1 \end{bmatrix},\qquad
\bfb_2 = \begin{bmatrix} \frac13\\-\frac13 \\ -1 \end{bmatrix},\qquad
\bfb_3 = \begin{bmatrix} 1\\1 \\ 1 \end{bmatrix}.
\]
This is a reparametrization of the twisted cubic curve, as a calculation yields
\[ \bfc\big((s+1)/2\big) = (s,s^2,s^3),\qquad s\in [-1, 1].\]

Clearly the half-turn $\bff(x,y,z) = (-x, y, -z)$ is a symmetry of the control polygon of $\bfc$, since $\bff(\bfb_i) = \bfb_{3-i}$ for $i = 0,1,2,3$. Hence Condition C1 holds for the curve~$\bfc$ with the symmetry~$\bff$ and M\"obius transformation $\varphi(t) = 1-t$. Taking the degree to be one in \eqref{eq:Bezier2} --- the simplest case not yielding a pipe surface --- and $a_1 = -a_0 = 1/2$, we obtain the radius function $r(t) = t - 1/2$, which satisfies Condition C2. The resulting symmetric canal surface $\SSS_{\bfc,r}$ is shown in Figure \ref{fig:examples}, right.
\end{example}

The curve defined by the parametrization $\bfc$ in \eqref{eq:Bezier} can also be used for defining global symmetric surfaces, as opposed to patches. Since $\bfc$ defines an irreducible algebraic variety, if $\bff\circ \bfc=\bfc\circ \varphi$ holds for $t\in [0,1]$, then $\bff\circ \bfc=\bfc\circ \varphi$ also holds for any $t\in \RR$ for which it is defined. Therefore, the parametrization \eqref{eq:Bezier} provides a global curve invariant under the symmetry $\bff$.

Notice that since the parametrization \eqref{eq:Bezier} is polynomial, the surface $\SSS$ produced this way is unbounded; bounded surfaces can be obtained from bounded spine curves, for instance represented by NURBS.

\subsection{Building symmetric blends}

Suppose that for two canal surfaces $\SSS_i=\SSS_{\bfc_i,r_i}$, with $i = 1,2$, we wish to construct a symmetric canal surface blend $\SSS$ connecting characteristic circles $k_i(t_i)\subset \SSS_i$, i.e., $\SSS\cap\SSS_i = k_i(t_i)$; see Figure \ref{fig:cylinders}. Recall from Remark \ref{rem:Minkowski} that $\SSS_i$ can be identified with a rational parametric curve $\big(\bfc_i(t);r_i(t)\big)$ in four-dimensional Minkowski space $\RR^{3,1}$. In \cite{Dahl2014} it is shown that if two curves in $\RR^{3,1}$ meet with $G^N$-continuity, with $N = 0,1,2$, then the same holds for the corresponding canal surfaces. Let us see how to apply this result to two different subproblems, by enforcing (the stronger) $C^N$-continuity on the curves in $\RR^{3,1}$, thus resulting in a $G^N$-continuous canal surface blend.

\subsubsection{Symmetry with respect to the plane of coplanar spine curves}
Suppose $\bff$ is a reflection in a plane $\Pi$ containing both spline curves $\bfc_1,\bfc_2$. We wish to construct a canal surface $\SSS = \SSS_{\bfc,r}$ that it is symmetric with respect to $\Pi$, i.e., invariant under $\bff$. This situation arises, for instance, when $\SSS_1,\SSS_2$ are surfaces of revolution with intersecting axes, which we wish to blend with a patch symmetric with respect to a plane containing the axes.

As explained in Example \ref{ex:planarsymmetry}, a patch is symmetric with respect to $\Pi$ whenever its spine curve $\bfc$ is contained in $\Pi$. We consider a B\'ezier curve~$\bfc$ as in \eqref{eq:Bezier} of degree $n = 2N + 1$, in which case the $C^N$-continuity conditions form the Hermite interpolation problem
\begin{equation}\label{eq:condit}
\begin{array}{llllll}
   \displaystyle \frac{n!}{(n-k)!}\Delta ^k {\bfb}_0     & = & \bfc^{(k)}(0)& =& \bfc^{(k)}_1(t_1)\\
   \displaystyle \frac{n!}{(n-k)!}\Delta ^k {\bfb}_{n-k} & = & \bfc^{(k)}(1)& =& \bfc^{(k)}_2(t_2)
\end{array},\qquad k = 0, \ldots, N,
\end{equation}
where $\Delta^k$ is the $k$-th forward difference operator as in \eqref{eq:forwarddifference}. Starting from $\bfb_0 := \bfc_1(t_1)$ and $\bfb_n = \bfc_2(t_2)$, the condition for $C^0$-continuity, the control points $\bfb_k, \bfb_{n-k}$ of $\bfc$ are determined from \eqref{eq:condit} for $k = 1, 2, \ldots, N$.  Notice in particular that no symmetry conditions are imposed on the $\bfb_i$. 

Additionally, as explained in Example \ref{ex:planarsymmetry}, any choice of the radius function~$r$ provides a canal surface symmetric with respect to $\Pi$. However, we need to impose continuity conditions on $r$ as well, resulting in the Hermite interpolation problem
\begin{equation}\label{eq:condit2}
r^{(k)}(0)= r^{(k)}_1(t_1),\qquad r^{(k)}(1)= r^{(k)}_2 (t_2), \qquad k=0,\ldots,N.
\end{equation}
A polynomial $r$ satisfying these conditions can be found as above, by expressing it in the Bernstein basis.

\begin{figure}
\begin{center}
\includegraphics[scale=0.25]{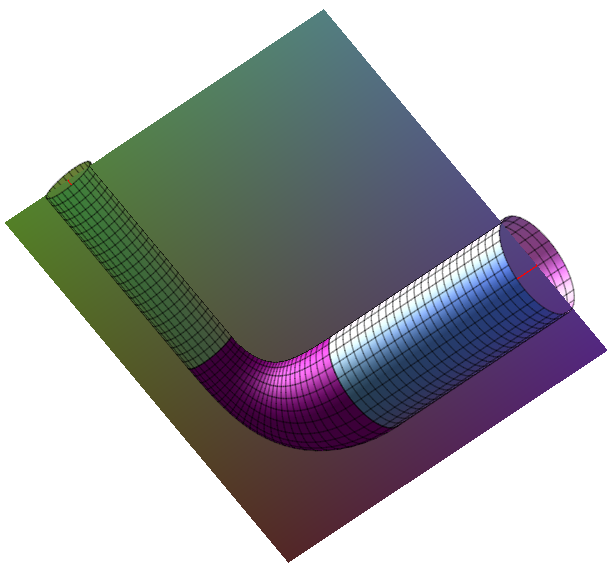}
\end{center}
\caption{The two cylinders $\SSS_1$ (right) and $\SSS_2$ (left) from Example \ref{ex:cylinders}, blended by the symmetric canal surface patch $\SSS$ (middle).}\label{fig:cylinders}
\end{figure}

\begin{example}\label{ex:cylinders}
Let us consider the two cylinders $\SSS_1$, $\SSS_2$ in Figure~\ref{fig:cylinders} with corresponding radii $r_1 = 1/2$, $r_2 = 1/4$ and axes (spine curves) 
\[ \bfc_1(t) = (0,0,-2t+1),\qquad \bfc_2(t) = (0,2t-1,0)\]
intersecting in the origin. Applying the results in this section, we let $t_1 = 0$, $t_2 = 1$ and connect $\bfc_1((-\infty,t_1])$, $\bfc_2([t_2,\infty))$ smoothly by a canal surface patch $\SSS$, shown in Figure~\ref{fig:cylinders}, with cubic B\'ezier spine curve $\bfc$ as in \eqref{eq:Bezier} and radius function $r(t) = \sum_{i=0}^3 a_i B_{i,3}(t)$. From \eqref{eq:condit} we determine
\begin{align*}
& \bfb_0 = \bfc(0) = \bfc_1(0) = (0,0,1),         & \qquad & \bfb_3 = c(1) = \bfc_2(1) = (0,1,0),\\
& \bfb_1 = \bfb_0 + \frac13 \bfc_1'(0)= (0,0,1/3),& \qquad & \bfb_2 = \bfb_3 - \frac13 \bfc_2'(1) = (0,1/3,0),
\end{align*}
which reduces to a quadratic spine curve $\bfc(t) = \big(0, t^2, (t-1)^2\big)$, while \eqref{eq:condit2} yields
\begin{align*}
& a_0 = r(0) = r_1(0) = 1/2, & \qquad & a_3 = r(1) = r_2(1) = 1/4,\\
& a_1 = a_0 + \frac13 r_1'(0) = 1/2, & \qquad & a_2 = a_3 - \frac13 r_2'(1) = 1/4,
\end{align*}
so that $r(t) = (2 - 3t^2 + 2t^3)/4$. Observe that the patch is symmetric with respect to the plane spanned by the axes of the cylinders.
\end{example}

\subsubsection{A more general case}

More generally, suppose that we are given two canal surfaces $\SSS_1$, $\SSS_2$ and a nontrivial isometry $\bff$ such that, for some integer $N$ and parameters $t_1, t_2$,
\begin{align}
\bff \circ \bfc^{(k)}_1(t_1) & = \bfc^{(k)}_2(t_2),     & k = 0,\ldots,N; \label{eq:C1}\\
              r^{(k)}_1(t_1) & = \pm(-1)^k r^{(k)}_2(t_2), & k = 0,\ldots,N. \label{eq:C2}
\end{align}
For instance, Conditions \eqref{eq:C1} and \eqref{eq:C2} hold for any integer $N$ and certain parameters $t_1, t_2$ when $\SSS_1,\SSS_2$ each have a unique spine curve and are related by an isometry~$\bff$. Under these conditions, there exists a symmetric patch (invariant under $\bff$) with degree of continuity equal to $N$. 

As before, we first choose a B\'ezier curve $\bfc$ of degree $n := 2N + 1$ that satisfies the continuity conditions \eqref{eq:condit}. Using \eqref{eq:C1} it follows that
\[ \Delta ^k {\bf b}_{n-k} = \frac{(n-k)!}{n!} \bfc^{(k)}_2(t_2) = \frac{(n-k)!}{n!} \bff \circ \bfc^{(k)}_1(t_1) = \bff \big( \Delta^k \bfb_0 \big), \]
implying that the B\'ezier curve $\bfc$ is invariant under $\bff$. Hence, as shown in Section~\ref{building-1}, $\bff\circ \bfc = \bfc \circ \varphi$ with $\varphi(t) = 1-t$. By Theorem \ref{thm:singlespine}, a canal surface with spine curve $\bfc$ is invariant under $\bff$ when, in addition, its radius function $r$ satisfies $r^2(t) = r^2(1-t)$, taking the form \eqref{eq:Bezier2}. Together with our assumption \eqref{eq:C2}, the continuity conditions \eqref{eq:condit2} become
\[ r^{(k)}(0) = r_1^{(k)} (t_1) = \pm (-1)^k r_2^{(k)} (t_2) = \pm (-1)^k r^{(k)} (1), \qquad k = 0,\ldots, N, \]
which is compatible with our invariance condition $r^2(t) = r^2(1-t)$. 

\begin{figure}
\begin{center}
\includegraphics[scale=0.265]{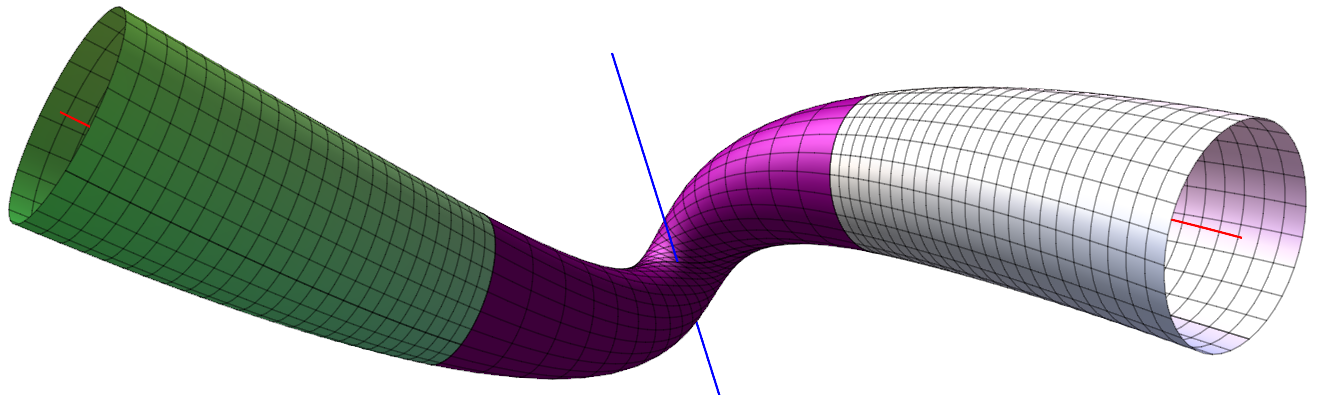}
\end{center}
\caption{The symmetric canal surface from Example \ref{ex:twisted}, with its middle singular part replaced by a regular canal surface patch with the same symmetry.}\label{fig:finalexample}
\end{figure}

\begin{example}
Consider the symmetric canal surface $\SSS$ from Example \ref{ex:twisted}. We wish to replace a piece surrounding the singularity by a regular canal surface patch, without breaking the symmetry. For this purpose, we do not modify the spine curve, but replace, in the interval $[t_1, t_2] = [0, 1]$, the linear radius function $r_1(t) = t - 1/2$ by a quadratic radius function $r(t) = \sum_{i=0}^2 a_i B_{i,2}(t)$. From \eqref{eq:condit2}, with $r_2 = -r_1$, imposing smoothness yields
\[ a_0 = r(0) = r_1(0) = - \frac12 = r_2(1) = r(1) = a_2, \]
\[ 0 = a_0 + \frac12 r_1'(0) = a_1 = a_2 - \frac12 r_2'(1) = 0. \]
Since the latter equation is consistent, we obtain a quadratic solution $r(t) = t^2 - t  + 1/2$. The resulting smooth piecewise canal surface is shown in Figure \ref{fig:finalexample}.
\end{example}

\section{Conclusion and final remarks}
We have presented results that characterize the symmetries of a canal surface $\SSS$ in terms of its spine curve and radius function. These results lead to an algorithm, which is easy to implement, for computing the symmetries of a canal surface with just one spine curve, and to a complete classification of the symmetries of Dupin cyclides. Moreover, the results lead to a method for constructing symmetric canal surfaces, and to strategies to build symmetric blends between two canal surfaces, in certain cases.

In this paper, we have assumed that $\SSS$ is defined by a rational spine curve $\bfc$ and a rational radius function $r$, in which case it is well known that $\SSS$ admits a rational parametrization as well. Conversely, and according to Proposition 1.1 in \cite{VL16}, if $\SSS$ is rational then $\bfc$ and $r^2$ must be rational too. However, there are rational canal surfaces where $r$ is not rational, but the square root of a rational function; see Example 3.2 in \cite{VL16}. The methods of Section \ref{sec:sym} can be adapted to this situation. 

Finally, in Remark \ref{rem:Minkowski} we observed that $\SSS$ can be identified with a curve in four-dimensional Minkowski space $\RR^{3,1}$. We have not yet explored the possibility of using this setup for computing or studying the symmetries of $\SSS$. In this sense, a potential idea is to use differential invariants of curves in $\RR^{3,1}$; see for instance~\cite{simsek}.

%%%%%%%%%%%%%%%%%%%%%%%%%%%%%%%%%%%%%%%%%%%%%%%%%%%%%%%%%%%%%%%%%%%%%%%%%%%%
\section*{References}
%%%%%%%%%%%%%%%%%%%%%%%%%%%%%%%%%%%%%%%%%%%%%%%%%%%%%%%%%%%%%%%%%%%%%%%%%%%%

\bibliographystyle{amsxport}

\begin{biblist}

\bib{AG16}{article}{
   author={Alc{\'a}zar, Juan Gerardo},
   author={Goldman, Ron},
   title={Finding the Axis of Revolution of an Algebraic Surface of Revolution},
   journal={IEEE Transactions on Visualization and Computer Graphics},
   volume={22},
   date={2016},
   number={9},
   pages={2082--93},
   %issn={0377-0427},
   %review={\MR{3197769}},
   %doi={10.1016/j.cagd.2014.02.004},
}

\bib{AHM15}{article}{
   author={Alc{\'a}zar, Juan Gerardo},
   author={Hermoso, Carlos},
   title={Involutions of polynomially parametrized surfaces},
   journal={Journal of Computational and Applied
Mathematics},
   volume={294},
   date={2016},
   %number={3-4},
   pages={23--38},
   issn={0377-0427},
   %review={\MR{3197769}},
   %doi={10.1016/j.cagd.2014.02.004},
}

\bib{Alcazar.Hermoso.Muntingh16a}{article}{
   author={Alc{\'a}zar, Juan Gerardo},
   author={Hermoso, Carlos},
   author={Muntingh, Georg},
   title={Similarity detection of rational space curves},
   note={To appear in Journal of Symbolic Computation},
   date={2016},
   eprint={http://arxiv.org/abs/1512.02620},
}

\bib{Alcazar.Hermoso.Muntingh16}{article}{
   author={Alc{\'a}zar, Juan Gerardo},
   author={Hermoso, Carlos},
   author={Muntingh, Georg},
   title={Detecting similarities of rational space curves},
	conference={
      title={Proceedings of the ACM on International Symposium on Symbolic and Algebraic Computation (2016)},},
			book={
      publisher={ACM},
  },
	pages={23--30},
  date={2016},
}

\bib{Alcazar.Hermoso.Muntingh15}{article}{
   author={Alc{\'a}zar, Juan Gerardo},
   author={Hermoso, Carlos},
   author={Muntingh, Georg},
   title={Symmetry detection of rational space curves from their curvature
   and torsion},
   journal={Comput. Aided Geom. Design},
   volume={33},
   date={2015},
   pages={51--65},
   issn={0167-8396},
   %review={\MR{3317263}},
   %doi={10.1016/j.cagd.2015.01.003},
}

\bib{Alcazar.Hermoso.Muntingh14}{article}{
   author={Alc{\'a}zar, Juan Gerardo},
   author={Hermoso, Carlos},
   author={Muntingh, Georg},
   title={Detecting symmetries of rational plane and space curves},
   journal={Comput. Aided Geom. Design},
   volume={31},
   date={2014},
   number={3-4},
   pages={199--209},
   issn={0167-8396},
   %review={\MR{3197769}},
   %doi={10.1016/j.cagd.2014.02.004},
}

\bib{Alcazar.Hermoso.Muntingh14b}{article}{
   author={Alc{\'a}zar, Juan Gerardo},
   author={Hermoso, Carlos},
   author={Muntingh, Georg},
   title={Detecting similarity of rational plane curves},
   journal={J. Comput. Appl. Math.},
   volume={269},
   date={2014},
   pages={1--13},
}


\bib{Chandru}{article}{
   author={Chandru, Vijaya},
   author={Dutta, Debasish},
   author={Hoffmann, Christoph M.},
   title={On the geometry of Dupin cyclides},
   journal={Visual Computer},
   volume={5},
   date={1989},
   pages={277--290},
   issn={0178-2789},
   %review={\MR{3317263}},
   %doi={10.1016/j.cagd.2015.01.003},
}


\bib{Cho1}{article}{
author={Cho, Hee Cheol},
author={Choi, Hyeong In},
author={Kwon, Song-Hwa},
author={Lee, Doo Seok},
author={Wee, Nam-Sook},
title={Clifford algebra, Lorentzian geometry
and rational parametrization of canal surfaces},
journal={Comput. Aided Geom. Design},
volume={21},
date={2004},
pages={327--339},
   issn={0167-8396},
	}
	
\bib{Cho2}{article}{
  author={Choi, Hyeong In},
  author={Lee, Doo Seok},
  title={Rational parametrization of canal surface by 4 dimensional Minkowski
Pythagorean hodograph curves.},
  conference={
      title={Proceedings of the Geometric Modeling and Processing 2000,
(GMP 00)},
      %address={ Washington, DC, USA},
      %date={2000},
  },
  book={
      %title={Proceedings of the Geometric Modeling and Processing 2000s},
      publisher={IEEE Computer Society},
      %series={Lecture Notes in Computer Science},
      %volume={1651},
  },
  pages={301--309},
  date={2000},
}

\bib{Coxeter69}{book}{
   author={Coxeter, Harold S.M.},
   title={Introduction to geometry},
   edition={2},
   publisher={John Wiley \& Sons, Inc., New York-London-Sydney},
   date={1969},
   pages={xviii+469},
   %review={\MR{0346644 (49 \#11369)}},
}

\bib{Dahl14}{thesis}{
   author={Dahl, Heidi E.I.},
   title={Improved blends between primitive surfaces},
   type={PhD thesis},
   date={2014},
   eprint={http://urn.nb.no/URN:NBN:no-48202}
}

\bib{Dahl2014}{inproceedings}{
    author = {Dahl, Heidi E.I.},
    booktitle = {Mathematical Methods for Curves and Surfaces},
    editor = {Floater, Michael and Lyche, Tom and Mazure, Marie-Laurence and M{\o}rken, Knut and Schumaker, Larry L.},
    pages = {88--111},
    volume = {8177},
    publisher = {Springer Berlin Heidelberg},
    series = {Lecture Notes in Computer Science},
    title = {{Piecewise rational parametrizations of canal surfaces}},
    year = {2014}
}

\bib{Degen02}{article}{
   author={Degen, Wendelin},
   title={Cyclides},
   book={
      title={Handbook of Computer Aided Geometric Design},
      publisher={North-Holland, Amsterdam},
   },
   date={2002},
   pages={575--601},
   %review={\MR{1928556}},
   %doi={10.1016/B978-044451104-1/50024-1},
}

\bib{Dupin}{book}{
   author={Dupin, Charles},
   title={Application de g\'eom\'etrie et de m\'echanique},
   publisher={Bachelier, Paris (France)},
   date={1822},
}

\bib{Dutta}{article}{
author={Dutta Debasish},
author={Martin Ralph R.}, 
author={Pratt Michael J.},
title={Cyclides in surface and solid modeling},
journal={IEEE Computer Graphics and Applications},
volume={13},
date={1993},
number={1},
pages={53--59},
issn={0272-1716}
}

\bib{Hilbert.Cohn-Vossen52}{book}{
   author={Hilbert, David},
   author={Cohn-Vossen, Stephan},
   title={Geometry and the imagination},
   note={Translated by P. Nem\'enyi},
   publisher={Chelsea Publishing Company, New York, N. Y.},
   date={1952},
   pages={ix+357},
   %review={\MR{0046650 (13,766c)}},
}

\bib{Jia}{article}{
author={Jia Xiaohong},
title={Role of moving planes and moving spheres following Dupin cyclides},
journal={Comput. Aided Geom. Design},
volume={31},
date={2014},
pages={168--181},
issn={0167-8396},
}

\bib{Krasauskas07}{article}{
   author={Krasauskas, Rimvydas},
   title={Minimal rational parametrizations of canal surfaces},
   journal={Computing},
   volume={79},
   date={2007},
   number={2-4},
   pages={281--290},
   issn={0010-485X},
   %review={\MR{2295528 (2008g:53004)}},
   %doi={10.1007/s00607-006-0204-0},
	}
	
\bib{Maxwell}{article}{
   author={Maxwell, James Clerk},
   title={On the Cyclide},
   journal={The Quarterly Journal of Pure and Applied Mathematics},
   volume={34},
   date={1867},
   %number={2-4},
   %pages={281--290},
   %issn={0010-485X},
   %review={\MR{2295528 (2008g:53004)}},
   %doi={10.1007/s00607-006-0204-0},
}


\bib{WebsiteGeorg}{article}{
   author={Muntingh, Georg},
   title={Personal Website},
   eprint={https://sites.google.com/site/georgmuntingh/academics/software}
}

\bib{Peternell.Pottman97}{article}{
   author={Peternell, Martin},
   author={Pottmann, Helmut},
   title={Computing rational parametrizations of canal surfaces},
   note={Parametric algebraic curves and applications (Albuquerque, NM,
   1995)},
   journal={J. Symbolic Comput.},
   volume={23},
   date={1997},
   number={2-3},
   pages={255--266},
   issn={0747-7171},
   %review={\MR{1448698 (98b:51024)}},
   %doi={10.1006/jsco.1996.0087},
}

\bib{Rotman95}{book}{
   author={Rotman, Joseph J.},
   title={An introduction to the theory of groups},
   series={Graduate Texts in Mathematics},
   volume={148},
   edition={4},
   publisher={Springer-Verlag, New York},
   date={1995},
   pages={xvi+513},
   isbn={0-387-94285-8},
   %review={\MR{1307623}},
   %doi={10.1007/978-1-4612-4176-8},
}

\bib{SanchezReyes}{article}{
  author={S\'anchez-Reyes J.},
  title={Detecting symmetries in polynomial B\'ezier curves},
  journal = {Journal of Computational and Applied
Mathematics},
  volume = {288},
  pages = {274--283},
  date = {2015},
}

\bib{Sendra.Winkler.Perez-Diaz}{book}{
   author={Sendra, J. Rafael},
   author={Winkler, Franz},
   author={P{\'e}rez-D{\'{\i}}az, Sonia},
   title={Rational algebraic curves},
   series={Algorithms and Computation in Mathematics},
   volume={22},
   %note={A computer algebra approach},
   publisher={Springer, Berlin},
   date={2008},
   pages={x+267},
   isbn={978-3-540-73724-7},
   %review={\MR{2361646 (2009a:14073)}},
   %doi={10.1007/978-3-540-73725-4},
}


\bib{simsek}{article}{
author={Sim\c{s}ek H.},
author={\"Ozdemir M.},
title={Similar and self-similar curves in Minkowski n-space},
journal={Bulletin of the Korean Mathematical Society},
volume={52},
date={2015},
number={6},
pages={2071--2093},
issn={1015-8634},
}

\bib{VL15}{article}{
   author={Vr\u{s}ek, Jan},
   author={L\'avi\u{c}ka, Miroslav},
   title={Determining surfaces of revolution from their implicit equations},
   journal={Journal of Computational and Applied Mathematics},
   volume={290},
   date={2015},
   number={2-4},
   pages={125--135},
   issn={0377-0427},
   %review={\MR{2295528 (2008g:53004)}},
   %doi={10.1007/s00607-006-0204-0},
}
		
\bib{VL16}{article}{
   author={Vr\u{s}ek, Jan},
   author={L\'avi\u{c}ka, Miroslav},
   title={Recognizing implicitly given rational canal surfaces},
   journal={Journal of Symbolic Computation},
   volume={74},
   date={2016},
   %number={2-4},
   pages={367--377},
   %issn={0377-0427},
   %review={\MR{2295528 (2008g:53004)}},
   %doi={10.1007/s00607-006-0204-0},
}	
	
\end{biblist}

\end{document}